\documentclass[11pt]{article}

\usepackage[a4paper, total={6in, 8in}]{geometry}
\usepackage[english]{babel}

\usepackage{graphicx}
\usepackage{wrapfig}
\usepackage{caption}
\usepackage{subcaption}
\usepackage{xcolor}

\usepackage{cite}

\usepackage{amsmath,amsthm,amssymb,mathtools}
\usepackage{hyperref}

\newcommand{\ubar}[1]{\text{\b{$#1$}}}

\usepackage{titlesec}
\titleformat*{\paragraph}{\itshape}
\titleformat*{\subsubsection}{\itshape}

\renewcommand{\P}{\mathbb{P}}
\newcommand{\Z}{\mathbb{Z}}
\newcommand{\A}{\mathcal{A}}
\newcommand{\B}{\mathcal{B}}
\newcommand{\C}{\mathcal{C}}
\newcommand{\E}{\mathcal{E}}
\newcommand{\F}{\mathcal{F}}
\newcommand{\Q}{\mathcal{Q}}

\newtheorem{theorem}{Theorem}
\newtheorem{lemma}{Lemma}

\theoremstyle{remark}
\newtheorem{remark}{Remark}

\usepackage[bottom]{footmisc}

\title{Crossing probabilities for planar percolation}
\author{Laurin Köhler-Schindler\footnote{ETH Zürich (laurin.koehler-schindler@math.ehz.ch, vincent.tassion@math.ethz.ch).} \and Vincent Tassion\footnotemark[1]}
\date{\today}

\mathtoolsset{showonlyrefs}

\begin{document}

\maketitle

\begin{abstract}
  We prove a general Russo--Seymour--Welsh result valid for any invariant planar  percolation process satisfying positive association. This means that  the probability of crossing a rectangle in the long direction is related by a homeomorphism to the probability of crossing it in the short direction. This homeomorphism is universal in the sense  that it  depends only on the aspect ratio of the  rectangle, and is uniform  in the scale and the considered model.
\end{abstract}

\section*{Introduction}
\vspace{-6pt}
We study a general bond percolation process on the square lattice $(\Z^2,\mathbb{E}^2)$.  A percolation configuration $\omega$  assigns to each edge $e$ a random status, \emph{open} if $\omega(e)=1$  or \emph{closed} if $\omega(e)=0$ (see Figure~\ref{fig:percolation-configuration} for an illustration). We consider a probability measure $\P$ on the space  of configurations $\{0,1\}^{\mathbb{E}^2}$ satisfying\vspace{-4pt}
    \begin{itemize}
    \item $\P$ is invariant under the symmetries of $\Z^2$ (translation, rotation, reflection);\vspace{-3pt}
    \item $\P$ is positively associated ($\P[\mathcal E\cap \mathcal F]\ge \P[\mathcal E]\cdot \P[\mathcal F]$ for all increasing events $\mathcal E,\mathcal F$).
    \end{itemize} \vspace{-3pt}
 We refer to Section \ref{section_background-and-notation} for precise definitions. The most famous example of such a measure is Bernoulli percolation, where edges are independently declared open with probability $p$  or closed with probability $1-p$. Many other models in statistical mechanics also satisfy the two above-mentioned properties, but involve dependencies between the status of edges. 

\begin{figure}
	\centering
	\begin{minipage}{.5\textwidth}
		\centering
		\includegraphics[width=.5\linewidth]{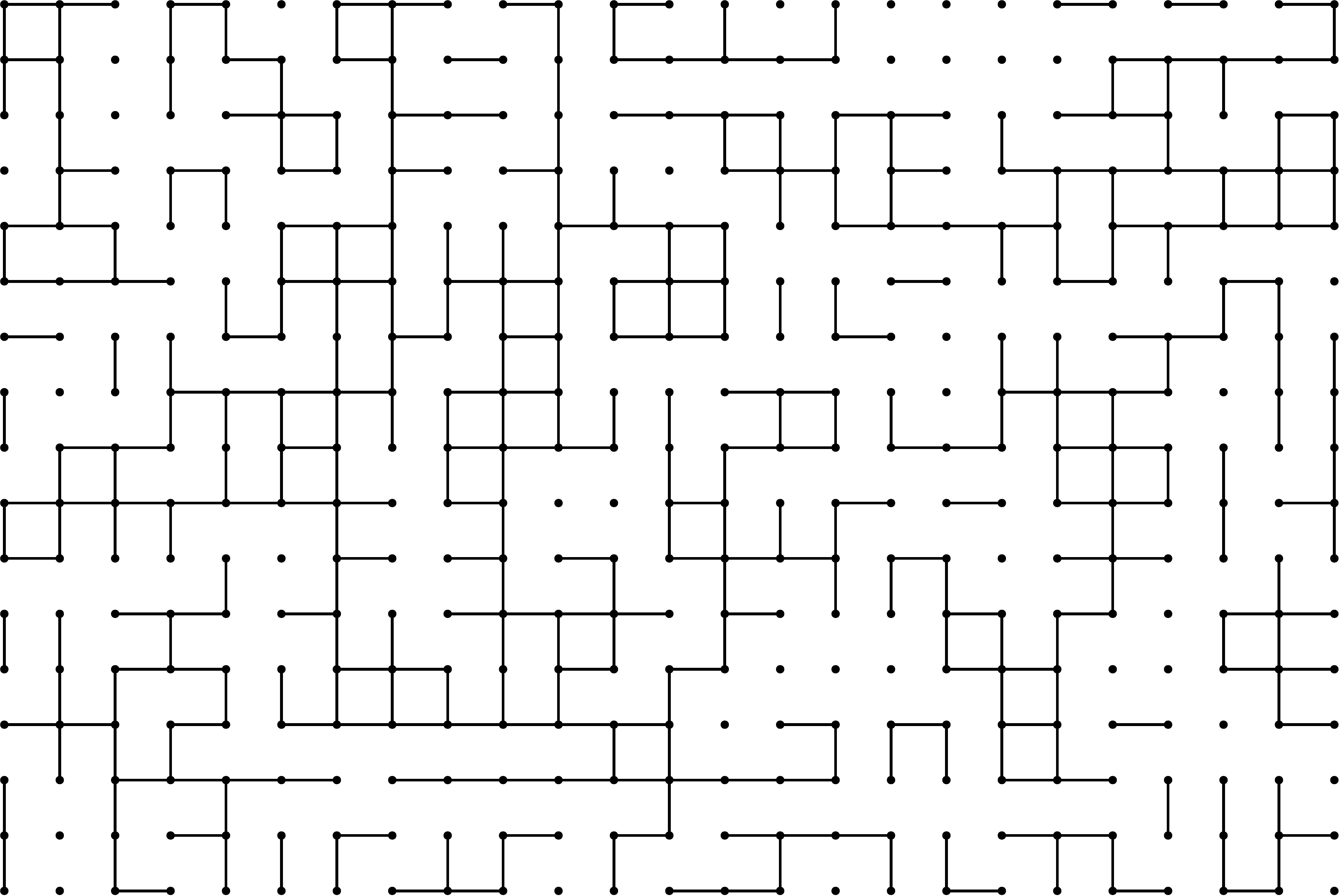}
		\captionof{figure}{A percolation configuration $\omega$}
		\label{fig:percolation-configuration}
	\end{minipage}%
	\begin{minipage}{.5\textwidth}
		\centering
		\includegraphics[width=.5\linewidth]{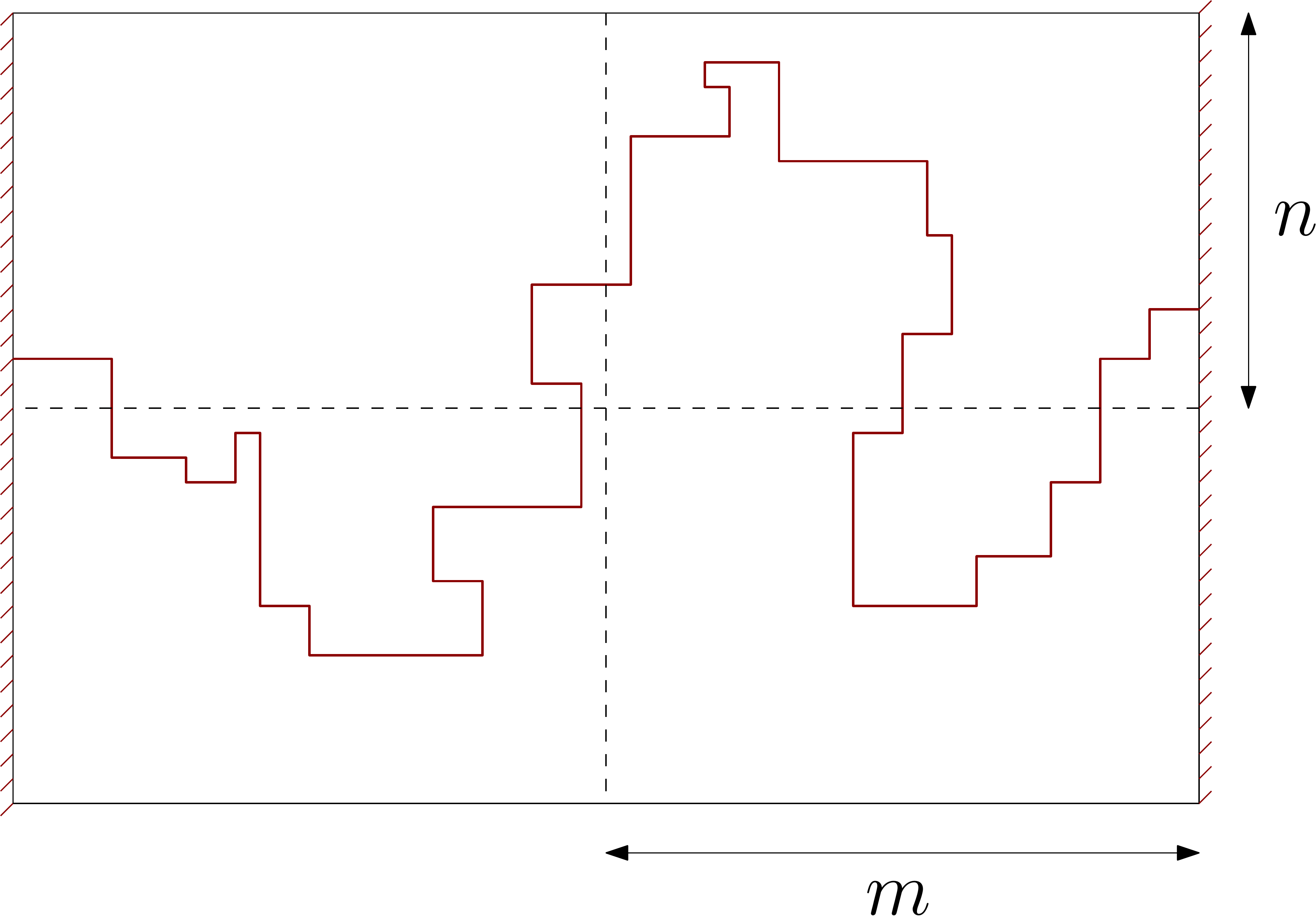}
		\captionof{figure}{The crossing event $\C(m,n)$}
		\label{fig:event_horizontal-crossing}
	\end{minipage}
\end{figure}
    
We are interested in the behaviour of crossing probabilities (illustrated in Figure \ref{fig:event_horizontal-crossing}, formally defined in Section \ref{section_background-and-notation}), which are probabilities of events of the form
    \begin{equation*}
      \label{eq:1}
      \C(m,n)=\left\{\text{
        \begin{minipage}{.55\linewidth}\centering
          There exists a path crossing from left to right in $[-m,m]\times[-n,n]$ made of open edges 
        \end{minipage}}\right\},\qquad m,n\ge1.
  \end{equation*}

Our main result shows a lower bound for crossing probabilities of rectangles in the long direction in terms of the short direction that is valid for any invariant, positively associated measure. 
\begin{theorem}
  For every $\rho \ge 1$, there exists a homeomorphism $\psi_\rho : [0,1] \to [0,1]$ such that for every invariant, positively associated measure $\P$ and for all $n\ge 1$,
  \begin{equation}
    \P[\C(\rho n,n)] \ge \psi_\rho \big(   \P[\C(n,\rho n)]\big).\label{eq:6}
  \end{equation} \label{thm_full-RSW}
\end{theorem}

\subsection*{Comments}  
\paragraph{1.  Robustness.}  While we show the theorem and its proof in the framework of bond percolation on $\mathbb Z^2$ for presentational purposes, we would like to point out  that the same proof applies to more general lattices (with sufficient symmetries) and more general processes (including site percolation, continuum models, or level lines of random fields). Furthermore the homeomorphism $\psi_\rho$ is universal, in the sense that it  does not depend on the considered model within  this general  class of  symmetric, positively associated percolation processes.

\paragraph{2. One scale relation.} We emphasize that the theorem makes no assumption on lower scales. Hence, in contrast to previous approaches using renormalization techniques, we obtain a relation between crossing probabilities that only involves one scale.

\paragraph{3. Finite-volume extensions.}    For several applications, the framework of an invariant measure on the whole lattice is too restrictive. In particular, in the study of models from statistical mechanics, it is more natural to work with a measure defined in a finite box. In  Section~\ref{section_extensions}, we discuss extensions  of our methods to this finite-volume framework, where one needs a replacement for the hypothesis of translation invariance.

\paragraph{4.  Motivations.}  The RSW theorem for Bernoulli percolation has been instrumental in the study of critical percolation (where it implies that crossing probabilities are bounded by  constants depending on the aspect ratio of the considered rectangle, independent of the scale)  and its near-critical regime (where it implies bounds on crossing probabilities below the correlation length). While the RSW result and its consequences have been extended to other models, most of the previously  known approaches involve positive association and symmetry together with an  additional assumption that is specific to the considered model.  In this sense, the theorem above unifies the previous results, and we hope that it will serve as a general tool to study critical and near-critical percolation processes in the future. 

\begin{wrapfigure}{r}{3cm}
	\centering
	\includegraphics[scale=0.2]{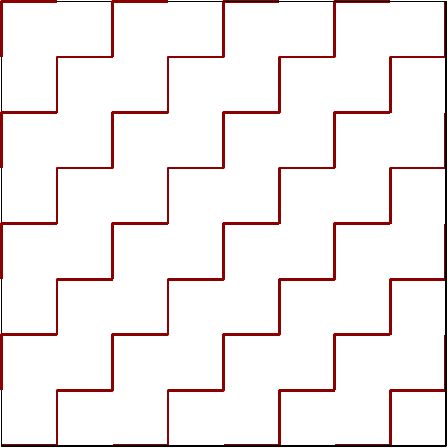}
	\caption{$\omega_{\text{diag}}$}
	\label{fig:counterexample}
\end{wrapfigure}
\paragraph{5. Minimal hypotheses.} If one removes the symmetry assumption, the statement is no longer valid. To see this, consider the deterministic configuration $\omega_{\text{diag}}$ made of infinite parallel ``diagonal lines'': such a model is positively associated, all the rectangles are crossed in the short direction, but no  rectangle is crossed in the long direction. The statement does not hold either for  general symmetric measure without the positive association assumption.  For a counter-example,  consider the symmetric percolation  defined by $\omega_{\text{diag}}$ with probability $1/2$, and its reflection  in the vertical axis otherwise.

\subsection*{Previous works}

In 1978, the first result of this kind was proven in the case of Bernoulli percolation by Russo \cite{russo1978note, russo1981critical} and independently by Seymour and Welsh \cite{seymour1978percolation}. The Russo-Seymour-Welsh (RSW) result for Bernoulli percolation has later been proven in different and shorter ways, for example in \cite{bollobas2006short} and in Smirnov's elegant proof which exploits symmetric domains and self-duality (see e.g. Section 3.1 of \cite{MR2523462}). It has quickly become a building block in the study of critical percolation (for example, it plays an important role in  the proof of  \cite{kesten1980critical} that $p_c=1/2$ for Bernoulli percolation on $\mathbb Z^2$, or in the proof of \cite{MR1851632} that critical percolation on the hexagonal lattice is conformally invariant) and near-critical percolation (see e.g. \cite{MR879034}, \cite{MR3790067}).

Consequently, RSW theory was extended to other planar percolation models, leading to the resolution of important conjectures. In the spirit of the original proof for Bernoulli percolation, exploration arguments have been developed for positively associated models satisfying a spatial Markov property, including continuum percolation in $\mathbb{R}^2$ with bounded radii \cite{roy1990russo,alexander1996rsw}, Bernoulli percolation on non-symmetric lattices \cite{MR2815602}, oriented percolation \cite{duminil2018box}, and FK-percolation with cluster-weight $q\ge 1$ \cite{duminil2011connection,beffara2012self,duminil2017continuity,duminil2019renormalization}. In the latter case, an additional difficulty arises due to the effect of boundary conditions and, additionnally to new RSW results,  those papers include new renormalization methods  to deduce uniform bounds at criticality. This leads to important breakthroughs for FK-percolation, such as the computation of the critical parameter  in \cite{beffara2012self}, or the rigorous identification of the order of the phase transition (see \cite{duminil2017continuity,duminil2016discontinuity,MR4150894}). More recently similar arguments were used to study the existence of macroscopic loops  in  loop $O(n)$-models in \cite{duminilcopin2020macroscopic} and the  fluctuations of random height functions in \cite{glazman2019uniform,duminilcopin2019logarithmic}.

An orthogonal and far-reaching approach has been developed by Bollobás and Riordan \cite{bollobas2006critical} who used a renormalization strategy to prove a weaker version of RSW for Voronoi percolation and to deduce $p_c=1/2$. Inspired by the new approach, a RSW result for positively associated models satisfying a quasi-independence assumption---such as Voronoi percolation---has been proven in \cite{tassion2016crossing}. It has been applied to continuum percolation in $\mathbb{R}^2$ with unbounded radii in \cite{ahlberg2018sharpness}, and used to study nodal lines of Gaussian fields in \cite{beffara2017percolation,beliaev2017russo,beliaev2018discretisation,rivera2019quasi} as well as Liouville first-passage percolation in \cite{ding2019liouville}. Without a quasi-independence assumption, a weaker version of RSW providing bounds along a subsequence of scales has been obtained in \cite{tassion2016crossing} and quantified in the appendix of \cite{muirhead2020phase}.

Finally, let us mention that RSW statements have also  been obtained  for percolation on two-dimensional slabs \cite{duminil2016absence,MR3729640,MR3707489,MR3955708}, and for planar percolations without positive association \cite{beffara2017withoutFKG,muirhead2020phase}.

In the spirit of \cite{bollobas2006critical,tassion2016crossing}, we prove Theorem \ref{thm_full-RSW} using a renormalization strategy, in the sense that we analyze the effect of a change of scale on crossing probabilities. We obtain new relations  on the crossing probabilities imposing that the RSW relation must hold. In contrast to previous renormalization approaches, these relations do not use independence properties, but rely on  positive association  and symmetry only.

The statement of Theorem~\ref{thm_full-RSW} was conjectured in \cite{duminil2016rsw} (see Conjecture 7.2). We refer to that paper for a more exhaustive background and motivations about  RSW-type results.
 
\subsection*{Organization of the paper}

In the next section, we introduce necessary background and notation. In Section~\ref{sec:sketch-proof}, we sketch the main steps of the proof. In Section~\ref{sec:reduct-relat-betw}, we give standard relations between certain connection  probabilities. The main new ingredient, the quasi-crossing event is introduced and discussed in  Section~\ref{sec:quasi-crossings}. The formal proof of the theorem is detailed in Section~\ref{sec:proof-main-theorem}, based on lemmas proved in the two previous sections. Finally, in Section \ref{section_extensions}, we discuss generalizations of the proof and provide details on extending the proof to finite-volume measures.

\subsection*{Acknowledgements}
This project has received funding from the European Research Council (ERC) under the European Union’s Horizon 2020 research and innovation program (grant agreement No 851565). 
Both authors are part of the NCCR Swissmap, and LKS also acknowledges the support of the SNF Grant $\#175505$.
We are grateful to Hugo Duminil-Copin and Ioan Manolescu for stimulating  questions and discussions that greatly inspired our study  of extensions to finite-volume measures. The notion of  symmetric stochastic domination leading to Theorem~\ref{thm_stochastic-domination} was suggested to us by Hugo Duminil-Copin, and   Theorem~\ref{thm_uniform-bound}  in Section~\ref{subsection_uniform-bound} answers a question of Ioan Manolescu, during a presentation of our work. 
We thank Hugo Vanneuville for comments on a preliminary version of this paper.

\section{Background and notation}
\label{section_background-and-notation}

On the square lattice $(\mathbb Z^2,\mathbb E^2)$, we consider a probability measure $\P$ on the space of percolation configurations $\{0,1\}^{\mathbb E^2}$. 

\paragraph{Symmetries.} 
We write $\Sigma$ for the group of symmetries of $\Z^2$ that is generated by translations by a vector $x \in \mathbb{Z}^2$, $\pi/2$-rotation around the origin, and vertical reflection along the $y$-axis. 
We assume that the measure $\P$ is invariant under the symmetries of $\Z^2$, meaning that for any event $\E$ and any symmetry $\sigma \in \Sigma$,
\begin{equation*}
  \P [\sigma\boldsymbol{\cdot}\E]=\P[\E],
\end{equation*}
where $\sigma\boldsymbol{\cdot}\E$ denotes the image of $\mathcal E$ under $\sigma$.      
Since we will regularly consider the translation of an event $\E$ by a vector $x \in \Z^2$, we write $x + \E$ for brevity.

\paragraph{Positive association.} An event $\mathcal E$ is called \emph{increasing} if its occurrence is favoured by open edges: $\omega \in \mathcal E$ and $\omega \le \omega'$ for the standard product ordering imply $\omega' \in \mathcal E$. 
We assume that increasing events $\mathcal E, \mathcal F$ are positively correlated under the measure $\P$,
\begin{equation*}
\P\left[\mathcal E \cap \mathcal F\right] \ge \P\left[\mathcal E \right] \cdot \P\left[\mathcal F\right],
\end{equation*}
which is usually referred to as the FKG inequality \cite{fortuin1971correlation}. Analogously, events favoured by closed edges are called \emph{decreasing}, and it is easy to see that decreasing events are also positively correlated. A direct consequence of positive association is the so-called square-root-trick (see, e.g.,\cite{grimmett1999percolation}), which states
\begin{equation*}
\max_{1 \leq i\leq k} \P[\E_i] \geq 1-(1-\P[\E])^{1/k}
\end{equation*}
 for increasing events $\E_1,\ldots,\E_k$  with $\E = \bigcup_{i=1}^{k} \E_i$.

\paragraph{Open paths.}
To obtain a relation between short and long crossings of rectangles, we carefully study the behaviour of open paths which form such crossings. To this end, we introduce the following notation. A \emph{path} $\gamma$ is a finite sequence of disjoint vertices $(\gamma_i)_{i=0}^{n} \subseteq \mathbb Z^2$ with $\gamma_{i-1}\gamma_{i} \in \mathbb E^2$ for all $1\le i\le n$, and it is said to be \emph{open} in a percolation configuration $\omega$ if $\omega(\gamma_{i-1}\gamma_{i}) = 1$ for all $1 \le i \le n$.
Naturally embedding the square lattice in the plane $\mathbb R^2$, we identify a path in $\mathbb Z^2$ with the piecewise linear, simple curve in $\mathbb R^2$ that is formed by straight lines between adjacent vertices along the path. Considering $D \subseteq \mathbb R^2$, we say that two subsets $A,B$ are \emph{connected} if there exists an open path in $D$ that starts in $A$ and ends in $B$, and the event that $A$ and $B$ are  connected is a typical instance of an increasing event.

\paragraph{Rectangle crossings}
For $m,n \ge 1$, we denote by $R(m,n)$ the rectangle $[-m,m]\times[-n,n]$, and we write $\C(m,n)$ for the event of a horizontal \emph{crossing} in $R(m,n)$, that is a connection from the left side $\{-m\}\times[-n,n]$ to the right side $\{m\}\times[-n,n]$.

\paragraph{Arms and bridges}
\begin{figure}
	\centering
	\begin{minipage}{.5\textwidth}
		\centering
		\includegraphics[width=.8\textwidth]{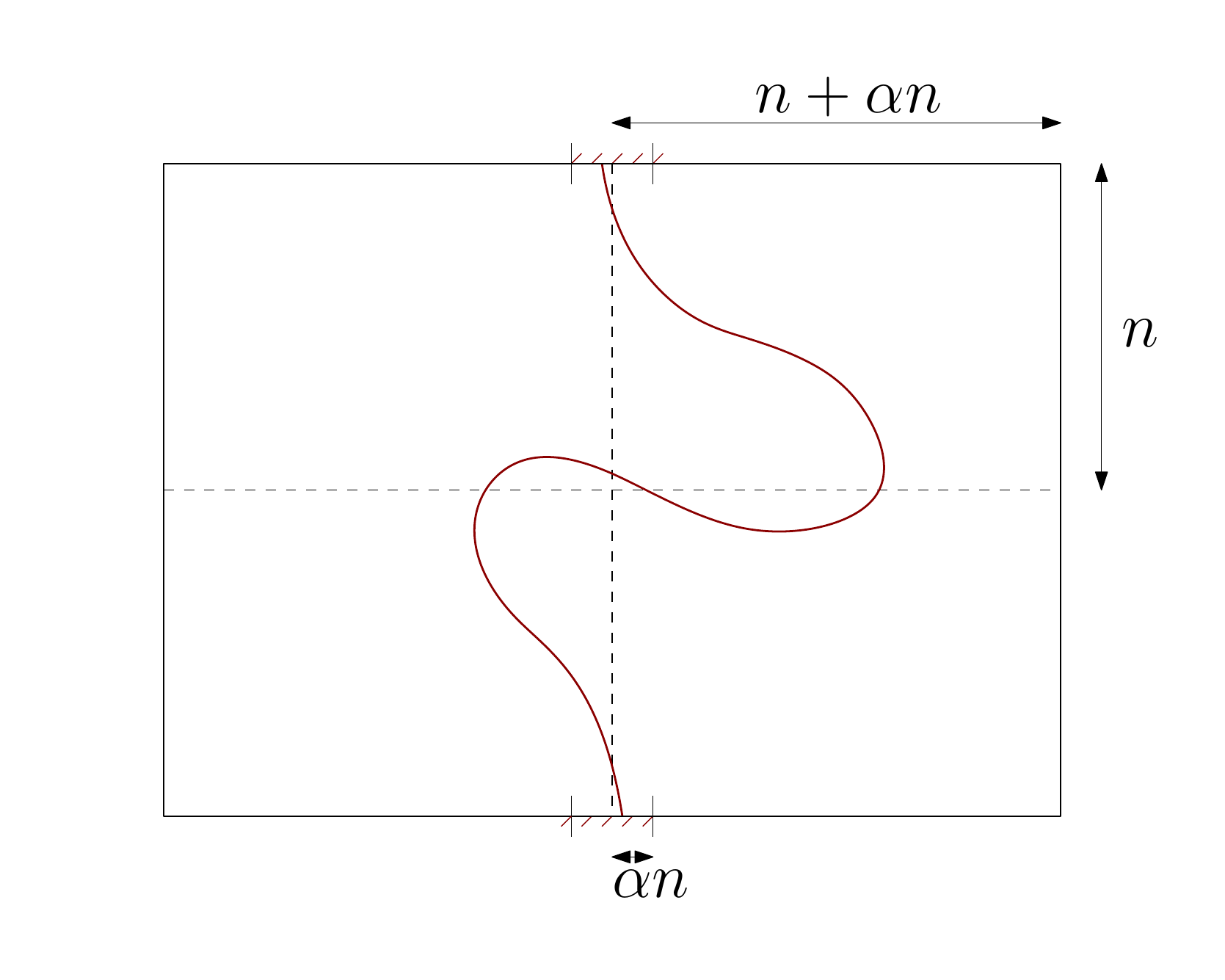}
		\captionof{figure}{The arm event ${\A(n)}$}
		\label{fig:event_arm}
	\end{minipage}%
	\begin{minipage}{.5\textwidth}
		\centering
		\includegraphics[width=.8\textwidth]{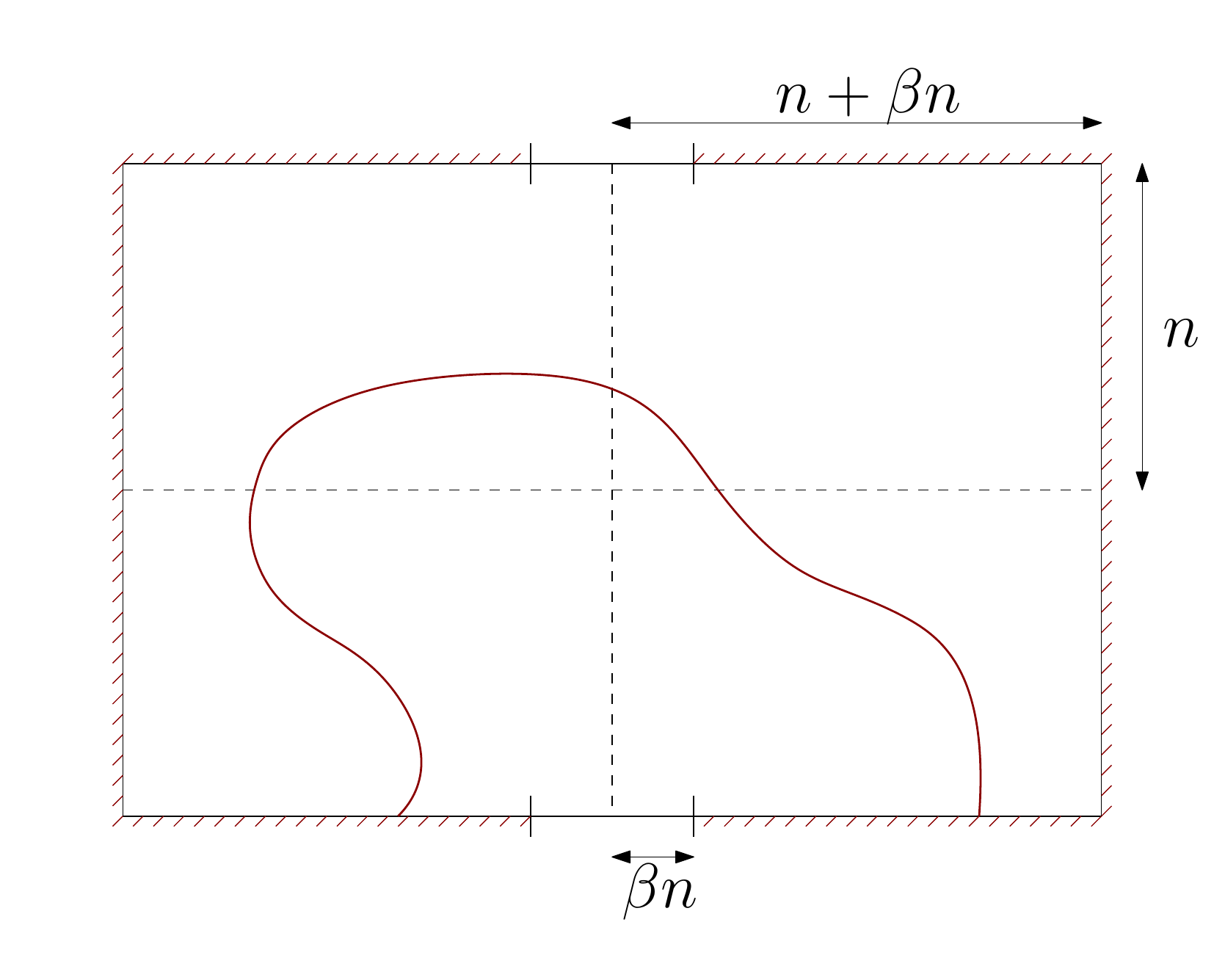}
		\captionof{figure}{The bridge event $\B(n)$}
		\label{fig:event_bridge}
	\end{minipage}
\end{figure}

Setting
\begin{equation}
	\alpha := 1/64 \quad \text{and} \quad \beta:= 1/12
	\label{eq:alpha-beta}
\end{equation}
for brevity, the \emph{arm event} $\A(n)$ is defined by a connection from the upper target $[-\alpha n, \alpha n] \times \{n\}$ to the lower target $[-\alpha n, \alpha n] \times \{-n\}$ in the short direction of the rectangle $R(n+\alpha n,n)$ (illustrated in Figure \ref{fig:event_arm}).  The \emph{bridge event} $\B(n)$ is defined by a connection from the left part $[-n-\beta n,-\beta n] \times \{-n,n\} \cup \{-n-\beta n\} \times [-n,n]$ to the right part  $[\beta n,n + \beta n] \times \{-n,n\} \cup \{n + \beta n\} \times [-n,n]$ in the long direction of the rectangle $R(n+\beta n,n)$ (illustrated in Figure~\ref{fig:event_bridge}).

\paragraph{Duality}
To study the absence of open paths, we consider the dual lattice $\left(\Z^2\right)^\star = (1/2,1/2) + \Z^2$ with edge set $\left(\mathbb{E}^2\right)^\star$. Every dual edge $e^\star \in \left(\mathbb{E}^2\right)^\star$ intersects a unique edge $e \in \mathbb{E}^2$, and we define the dual percolation configuration $\omega^\star$ by $\omega^\star(e^\star)=1-\omega(e)$. In analogy to the previous paragraph, a path $\gamma = (\gamma_i)_{i=0}^{n} \subseteq \left(\mathbb Z^2\right)^\star$ is called \emph{dual open} if $\omega^\star(\gamma_{i-1}\gamma_i)=1$ for all $1\le i\le n$, and this gives rise to the notion of \emph{dual connected} subsets.
It is important to notice that  the dual probability measure on $\{0,1\}^{\left(\mathbb{E}^2\right)^\star}$ induced by the primal measure $\mathbb P$ is also invariant under symmetries and positively associated. In particular all the statements that we prove  for the primal measure have a dual analogue.

We make use of the following standard fact in planar percolation: in the rectangle $R(m,n)$ for integers  $m$, $n$, the left is  connected to the right  if and only if the top  is \emph{not} dual connected to the bottom. 
We emphasize that this fact does not rely on self-duality of the square lattice. It holds in greater generality, and we refer to \cite{grimmett1999percolation,bollobas2006percolation} for more background on the use of planar duality in percolation theory.
Combined with the invariance under symmetries, this fact directly implies that for every $m,n$ integers,
\begin{equation}
	\label{eq:7}
	\mathbb P[\mathcal C(m,n)]=1-\mathbb P[\mathcal C^\star(n,m)],
\end{equation}
where $\C^\star(n,m)$ denotes a horizontal \emph{dual crossing} in the $R(n,m)$, that is a dual connection from the left side to the right side.

\section{Sketch of proof}
\label{sec:sketch-proof}

In this section we present the main ideas of the proof of Theorem~\ref{thm_full-RSW}.
A first idea that comes to mind to obtain the bound \eqref{eq:6} is to use positive association and symmetries to  ``glue'' short crossings together in order to obtain a long crossing. This strategy (that we will refer to as the ``standard gluing'') works very well if one starts with a lower bound on long crossings and wishes to obtain a lower bound on the probability of  even longer crossings. If one starts with  short crossings instead, the strategy fails, because one cannot exclude the case when  crossings are ``typically'' very tortuous and look like Peano curves (see \cite{bollobas2006percolation} for more details on this phenomenon).
Contrary to previous approaches, we will not try to rule out this behaviour, but rather use it in our favour. More precisely, if crossings at scale $n$ typically meander, we will define a well-chosen  lower scale $m_0\le n$ (possibly much smaller) that they typically reach, and at which we will be able to glue them to each other.

 In order to keep the focus on the ideas, we do not sketch the proof of Theorem~\ref{thm_full-RSW} in full generality, but we explain the proof of the following weaker statement: a  uniform bound on the crossings in the short direction implies a uniform bound on the crossing probabilities in the long direction, namely
\begin{equation}
  \label{eq:2}
  (\forall \rho>1\quad \inf \mathbb P[\mathcal C(n,\rho n)]>0) \implies (\forall \rho>1\quad \inf \mathbb P[\mathcal C(\rho n,n)]>0).
\end{equation}
By standard gluing constructions (see Section~\ref{sec:reduct-relat-betw}),  short crossings  can be easily related to the arm events introduced in the previous section, and long crossings to bridge events.  Therefore it suffices to prove  the following implication between arm and bridge events
\begin{equation}
  \label{eq:3}
  \inf \mathbb P[\mathcal A(n)]>0\implies \inf \mathbb P[\mathcal B(n)]>0.
\end{equation}
This reduction of the RSW problem to arms and bridges is not new and was already used in several previous works on the topic.

We now give more details on how to prove \eqref{eq:3}.
Let us fix a constant $c_0>0$ such that 
\begin{equation}
  \label{eq:9}
  \forall n\ge1 \quad \mathbb P[\mathcal A(n)]\ge c_0.
\end{equation}
Our goal is then to show that the bridge probability $\mathbb P[\mathcal B(n)]$  is bounded below by some other constant $c_3>0$ that depends on $c_0$ only. (Below we will introduce intermediate constants $c_0>c_1>c_2>c_3>c_4>0$ which by convention do not depend on  the considered scales $m$ and $n$).

In order to  implement the strategy described at the beginning of the section, we need to formalize the fact that ``crossings at scale $n$ reach scale  $m$'', which is achieved by  the notion of \emph{quasi-crossing}. Setting $\beta=1/12$, it is defined by
\begin{equation}
  \label{eq:8}
  \mathcal Q(m,n)=
  \begin{minipage}[c]{.3\linewidth}
    \includegraphics[width=\linewidth]{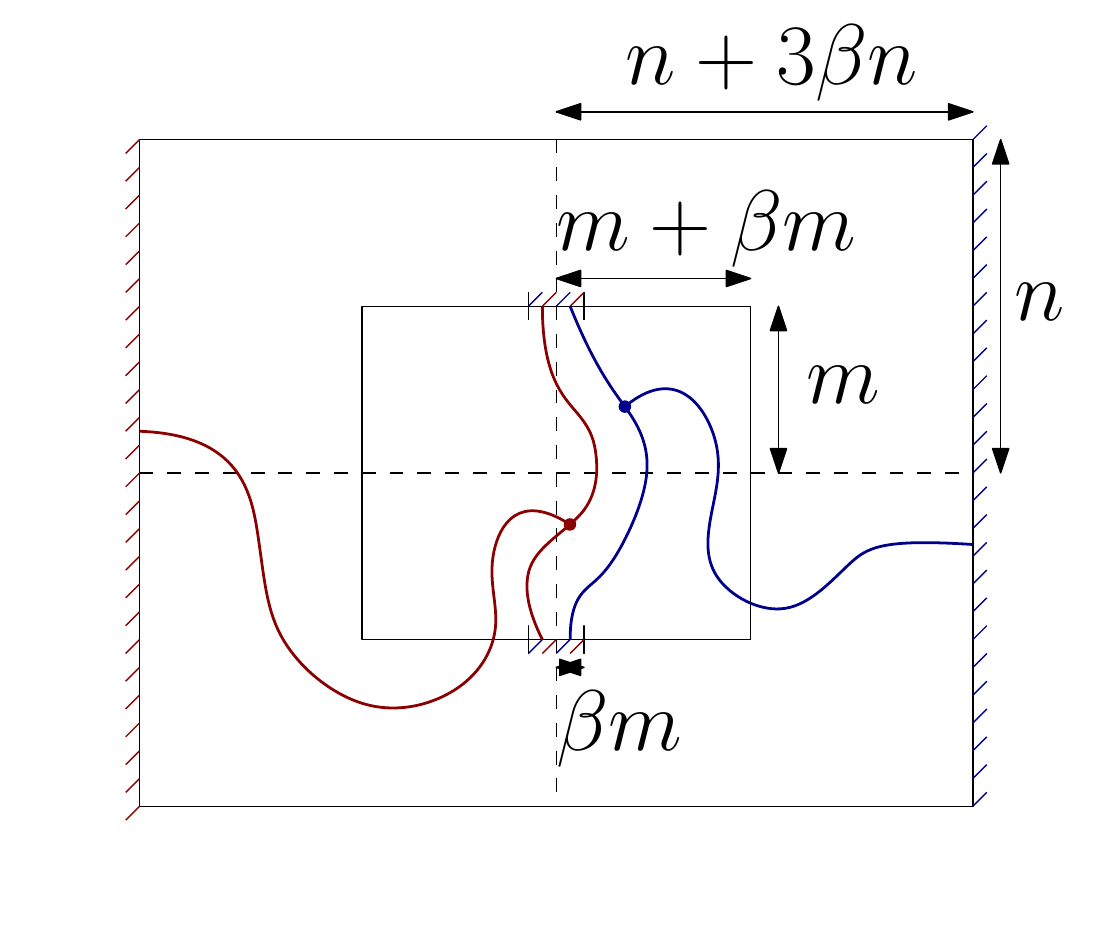}
  \end{minipage}
\end{equation}
where the picture  represent  two arms\footnote{Formally, this is not exactly the notion of arm as defined before, since the target regions have size $\beta m$ instead of $\alpha m$, and the rectangle has a slightly different aspect ratio.} at scale $m$, which  are respectively connected to the left and right side of $R(n+3\beta n,n)$ (see Section~\ref{sec:quasi-crossings} for a more precise definition).
This quasi-crossing event is instrumental in our approach and we list below its three main properties (\textbf{P1-3}), under the hypothesis~\eqref{eq:3}.  

First, for every $n\ge1$ we have
  \begin{equation}
    \label{eq:10}\tag{\textbf{P1}}
    \mathbb P[\mathcal Q(4n,n)]\ge c_1. 
  \end{equation}
  This corresponds to the idea that we can always assume that  crossings at scale $4n$ reach scale $n$. The proof of this fact involves a standard gluing construction: in order to obtain the event  $Q(4n,n)$, we  use symmetric versions of arm events $\mathcal A(m)$ at intermediate scales $m \in [n,4n]$.
  
 The key property of the quasi-crossing, that we call the cascading property is the following alternative: for every $1\le m \le m' $, we have
  \begin{equation}
    \label{eq:11}\tag{\textbf{P2}}
   \left.
     \begin{array}[c]{ll}
       \text{either} & \mathbb P[\mathcal B(m')]\ge c_3,\\
       \text{or}& \mathbb P[\mathcal Q(m',m)]\ge c_2\implies \mathbb P[\mathcal Q(4m',m)]\ge c_2.
     \end{array}\right.
  \end{equation}  
  This  statement illustrates how  we use the ``fractality'' of open paths, whenever we do not have a lower bound on the bridge event at scale $m'$: in this case, the scale $m$ cannot \emph{only} be reached from scale $m'$, but also from the scale $4m'$.

Finally it follows directly from the definition of $\mathcal  Q(n,m)$ and positive association that
  \begin{equation}
    \label{eq:12}\tag{\textbf{P3}}
    \mathbb P[\mathcal B(n)]\ge \mathbb P[\mathcal Q(n,m)] \mathbb P[\mathcal B(m)].
  \end{equation}
  In words, one can use a bridge at scale $m$  to ``close'' a quasi-crossing and obtain a bridge at scale $n$. This allows us to ``transport'' a lower bound on the bridge from scale $m$ to scale~$n$.

We now explain how to conclude the proof from these three properties. Let $n\ge 1$.
By considering the largest scale below $n$ where the bridge event has probability larger than $c_3$, we can define a scale $m_0$ such that 
\begin{equation}
  \label{eq:17}
  \mathbb P[\mathcal B(m)]
  \begin{cases}
    \ge c_3& m=m_0,\\
    < c_3 &  m_0<m\le n,
  \end{cases}
\end{equation}
where $c_2$ is the constant appearing in Property \eqref{eq:11} above. If $n=m_0$, we are done, and we can assume without loss of generality that $1\le m_0<n$.  Then, using first \eqref{eq:10}, and then the cascading property \eqref{eq:11}, one can prove by induction that crossings at scale $n$ reach the scale $m_0$ with a good probability:
\begin{equation}
  \label{eq:15}
  \mathbb P[\mathcal Q(n,m_0)]\ge c_2.
\end{equation}
Finally the closing property \eqref{eq:12} concludes that
\begin{equation}
  \label{eq:18}
  \mathbb P[\mathcal B(n)]\ge c_4:=c_2c_3. 
\end{equation}

In order to prove the more general statement of Theorem~\ref{thm_full-RSW}, we will prove three statements in Lemmas~\ref{lemma:second-relation}--\ref{lemma:fourth-relation} which generalize respectively the three properties \textbf{(P1-P3)} presented above by replacing the constants by expressions involving increasing  homeomorphisms. The strategy of the proof  follows exactly the strategy presented above, with additional technicalities arising from the absence of uniform lower bound on the arm events.

\section{From rectangle crossings to bridges and arms}
\label{sec:reduct-relat-betw}

In this section, we relate short crossings to arm events and long crossings to bridge events, using standard gluing constructions. We  introduce the notation
\begin{equation*}
	a(n):=\P\big[\A(n)\big] \quad \text{and} \quad b(n):=\P\big[\B(n)\big],
\end{equation*}
and summarize the relations that will be used to prove the main theorem  in Section \ref{sec:proof-theorem}.

\begin{lemma}\label{lemma:first-relation}
	Set $f_1(x):=(1-(1-x)^{1/2000})^{2000}$. For all $n \ge 1$, we have
	\begin{itemize}
		\item[(i)] $\P\left[\C(8 n, n)\right] \ge f_1\left(b(n)\right)$,
		\item[(ii)] $a(n) \ge f_1\left(\P\left[\C( n, 8n)\right]\right)$,
		\item[(iii)] $\P\left[\C(\rho n, n)\right] \ge \P\left[\C(2n, n)\right]^{2\rho-1}$ for every integer $\rho \ge 1$.
	\end{itemize}
\end{lemma}

\begin{figure}
	\centering
	\begin{minipage}{.5\textwidth}
		\centering
		\includegraphics[width=.8\textwidth]{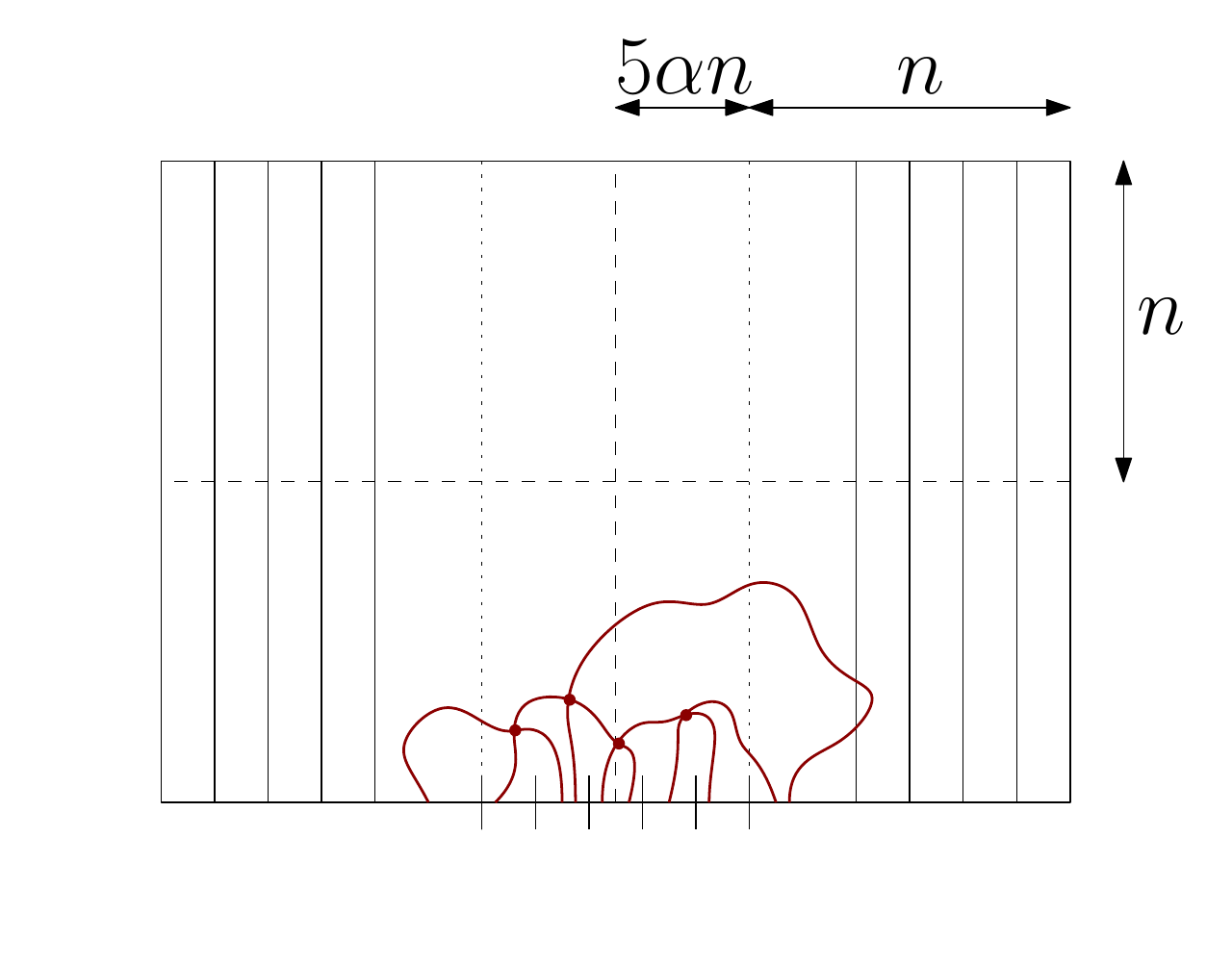}
	\end{minipage}%
	\begin{minipage}{.5\textwidth}
		\centering
		\includegraphics[width=.75\textwidth]{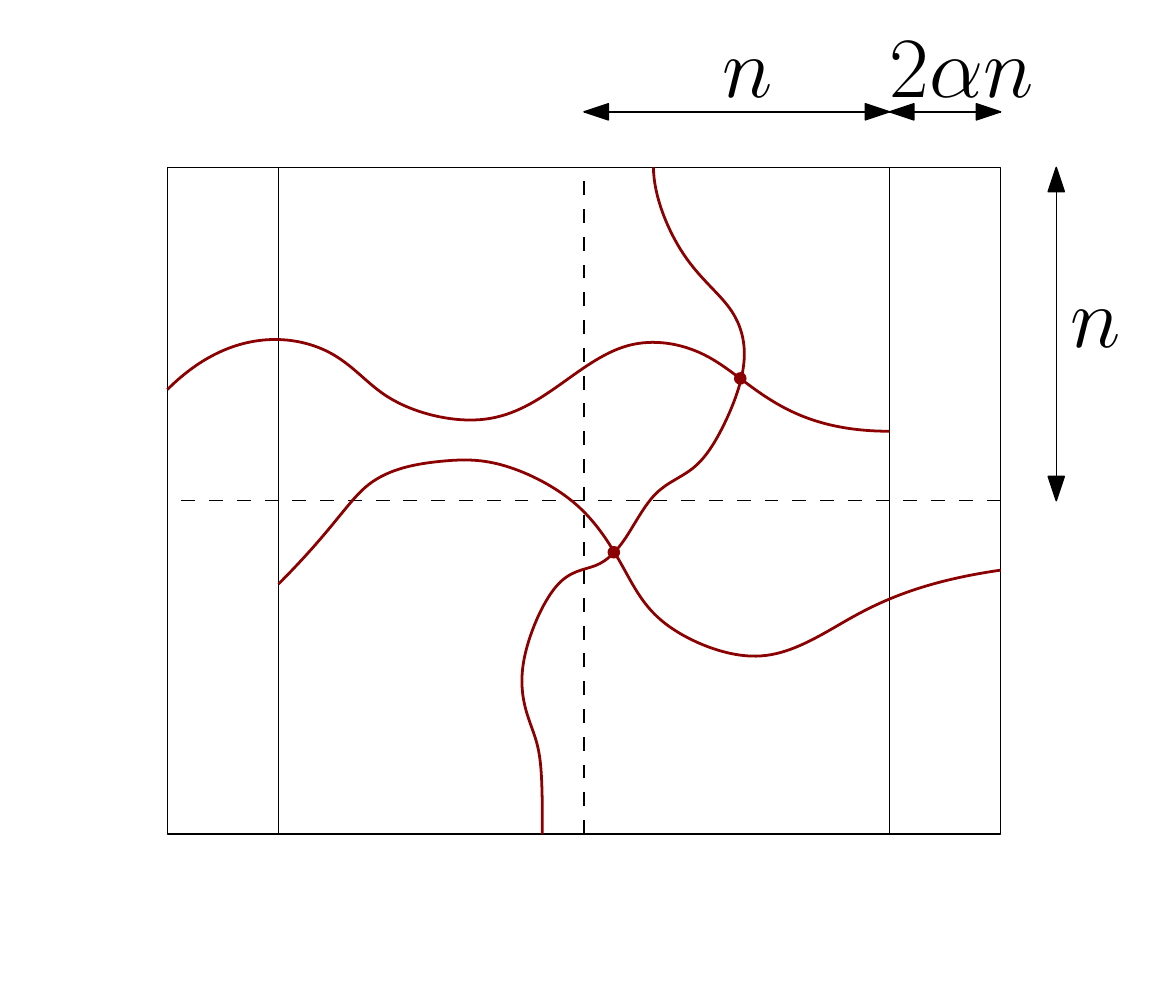}
	\end{minipage}
	\caption{Illustrations of gluing constructions used in the proof of Lemma \ref{lemma:first-relation}. Left: intersection of five translated versions of the event $\F$. Right: intersection of two translated versions of the event $\C(n+\alpha n,n)$ with a top-down crossing in $R(n,n)$. }
	\label{fig:illustration-prop1}
\end{figure}

\begin{proof}
	For simplicity, let us assume that $\alpha n$ is an integer. We define the following variation of the bridge event. Let $\widetilde{\B}(n)$ be defined by a connection from the left part $[-n-\alpha n,-\alpha n] \times \{-n,n\} \cup \{-n-\alpha n\} \times [-n,n]$ to the right part  $[\alpha n,n + \alpha n] \times \{-n,n\} \cup \{n + \alpha n\} \times [-n,n]$ in the rectangle $R(n+\alpha n,n)$. Analogously the event $\widetilde{\B}^\star(n)$ denotes a dual connection from the left part to the right part. 
	
	Below, we will prove for all $n\ge 1$,
	\begin{equation}\label{eq:23}
		\P\left[\C(8n,n)\right] \ge f_1\left(\P[\widetilde{\B}(n)]\right). 
	\end{equation}
	This implies \textit{(i)} by monotonicity since $\P[\widetilde{\B}(n)] \ge \P[\B(n)]$. To obtain \textit{(ii)}, we use its dual analogue 
        \begin{equation}
          \label{eq:30}
          \P[\C^\star(8n,n)] \ge f_1\left(\P[\widetilde{\B}^\star(n)]\right).
        \end{equation} By the duality relation \eqref{eq:7}, $\P[\C^\star(8n,n)] = 1- \P[\C(n,8n)]$ and similarly, it holds that $\P[\widetilde{\B}^\star(n)] = 1- \P[\A(n)]$. 	
	Combining these two equalities with \eqref{eq:30} concludes since $f_1(x) = 1-f_1^{-1}(1-x)$ for $x\in[0,1]$.

	 To prove \eqref{eq:23}, we first distinguish possible realizations of the event $\widetilde{\B}(n)$ and then conclude in each case using basic gluing constructions. 
	By the square-root-trick and invariance under symmetries, it holds that 
	\begin{equation*}
		\max \Big\{\mathbb{P}\left[\E\right],\mathbb{P}\left[\C(n+\alpha n,n)\right] \Big\} \ge 1 - \left(1-\P\left[\widetilde{\B}(n)\right]\right)^{1/5},
	\end{equation*}
	where $\E$ denotes a connection in the rectangle $R(n+\alpha n)$ from the left part $[-n-\alpha n,-\alpha n] \times \{-n,n\} \cup \{-n-\alpha n\} \times [-n,n]$ to the bottom-right boundary segment $[\alpha n,n + \alpha n] \times \{-n\}$. 
	
	In the first case, we deduce that $\P\left[\F\right] \ge \P\left[\E \cap (v \boldsymbol{\cdot} \E)\right] \ge (1 - (1- \P[\widetilde{\B}(n)])^{1/5})^2$, where $v \boldsymbol{\cdot} \E$ denotes the vertical reflection of $\E$ along the y-axis and  $\F$ denotes a connection in  $R(n + \alpha n,n)$ from the bottom-left  $[-n-\alpha n,-\alpha n] \times \{-n\}$ to the bottom-right boundary segment $[\alpha n,n + \alpha n] \times \{-n\}$. It now suffices to intersect $8/\alpha$ horizontally translated versions of the event $\F$ as illustrated in Figure \ref{fig:illustration-prop1} to guarantee the occurrence of the long crossing $\C(8 n,n)$. Noting that $16/\alpha \le 2000$ and using positive association, this establishes \eqref{eq:23} in this case. 
	
	In the second case, we combine long crossings from left to right in $[-n-2\alpha n, n]\times[-n,n]$ and $[-n, n+2\alpha n]\times[-n,n]$ with a crossing from top to bottom in $[-n,n]^2$ to construct the event $\C(n+2\alpha n, n)$ (see Figure \ref{fig:illustration-prop1}). As a consequence of invariance under symmetries and positive association, we obtain $\P[\C(n+2\alpha n,n)] \ge \P[\C(n+\alpha n,n)]^3$. Inductively, we deduce that $\P[\C(8 n,n)] \ge \P[\C(n + \alpha n,n)]^{2k -1}$ for $k = 7/\alpha$. This concludes the proof of \eqref{eq:23} since $14/\alpha \le 2000$.
	
	Finally, to obtain \textit{(iii)}, we use the same gluing construction as in the second case (with $\alpha$ replaced by $1$).
\end{proof}
\color{black}

\section{Quasi-crossings}
\label{sec:quasi-crossings}

\subsection{Definition}

\begin{figure}
	\centering
	\includegraphics[width=.41\linewidth]{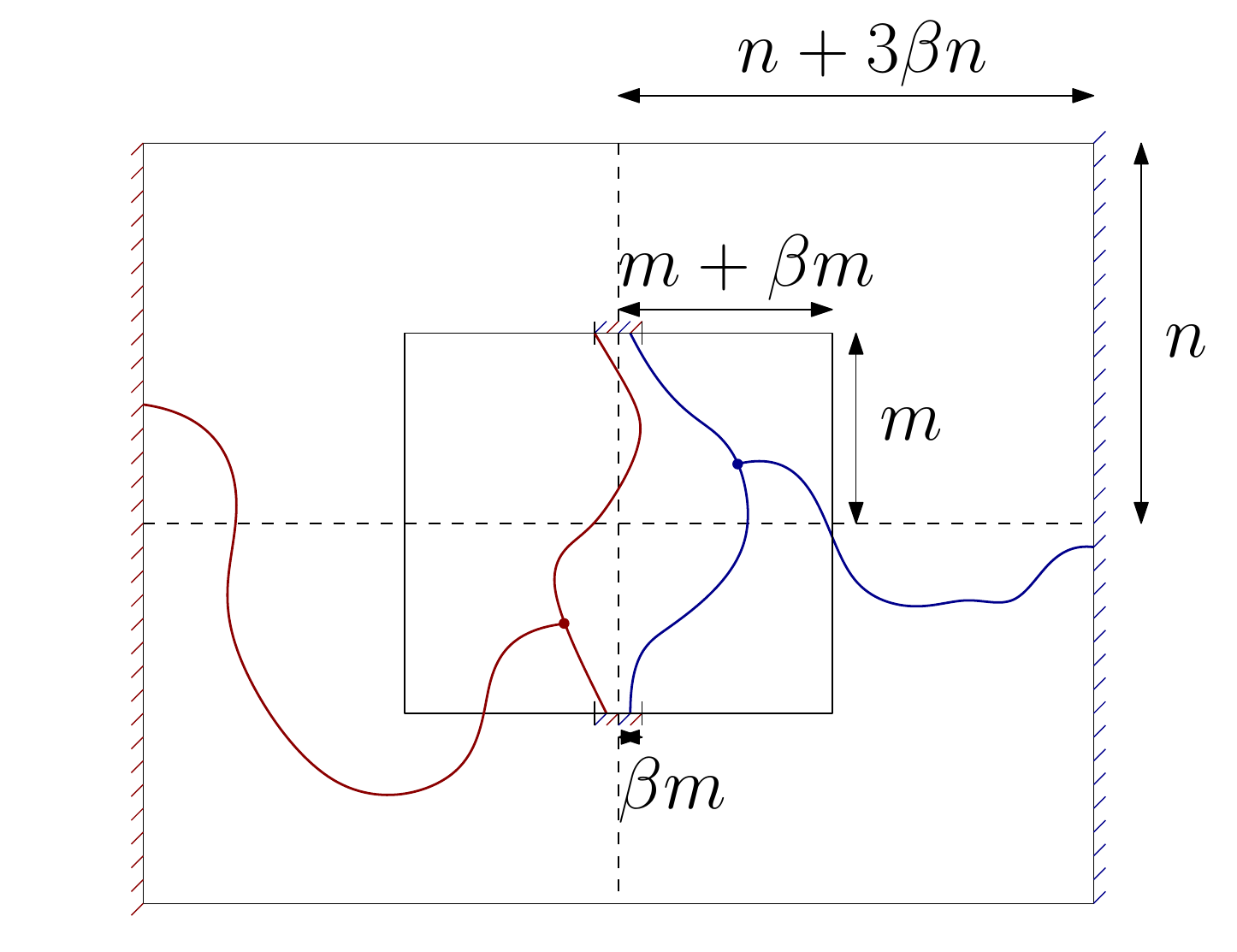}
	\caption{The quasi-crossing event $\mathcal{Q}(n,m)$}
	\label{fig:quasi-crossing}
\end{figure}

Here and subsequently, an \emph{$m$-path} is a path in the rectangle $R(m+\beta m,m)$ that starts from the upper target $[-\beta m, \beta m]\times\{m\}$ and ends at the lower target $[-\beta m, \beta m]\times\{-m\}$.
At scale $n\ge 1$, the \emph{quasi-crossing} $\mathcal{Q}(n,m)$ connecting down to scale $m\le n$ is defined as the following event (see Figure \ref{fig:quasi-crossing}): in the rectangle $R(n + 3\beta n,n)$, there exist an open $m$-path that is connected to the  left side $\{-n - 3\beta n\}\times[-n,n]$ and an open $m$-path that is connected to right side $\{n + 3\beta n\}\times[-n,n]$. 
We denote its probability by 
\begin{equation*}
	q(n,m):=\P\big[\mathcal{Q}(n,m)\big].
\end{equation*}
The event $\mathcal Q(n,m)$ formalizes the idea that crossings at scale $n$ reach down to scale $m$.

\subsection{A topological lemma }
\label{sec:topological-lemma-}

In this section, we prove a deterministic lemma from planar topology, that will be useful when working with quasi-crossings.  One can think of the quasi-crossing as the existence of a ``corridor with two walls'': the left wall is made of the union of the left $m$-path and the path connecting it to the left side of the large rectangle, and the right wall is its reflected analogue. Then, if one intersects the event $\Q(n,m)$ with the existence of a top-down crossing, this top-down crossing must go through the ``corridor'' or it must hit one of the walls.

\begin{lemma}[Corridor construction]
  \label{lem:1}
  Let $m\ge1$ and consider two rectangles $R:=R(m,m+\beta m)$ and $S$  with $R\subset S$. 
  Let $\gamma,\, \gamma'$ be two $m$-paths, and $W$ (resp. $W'$) be the union of $\gamma$ (resp. $\gamma'$) together with  a path in $S$ connecting $\gamma$ (resp. $\gamma'$) to the left (resp. right) side  of $S$. Then every path  from top to bottom in $S$ must 

  -  contain an $m$-path,

  - or intersect the structure $W\cup W'$.
\end{lemma}

\begin{remark}\label{rem:1}
  The lemma above can be generalized in several ways. In particular, we will also use it  with the roles of left-right and top-down  being switched.
\end{remark}

\begin{figure}
	\centering
	\begin{minipage}{.5\textwidth}
		\centering
		\includegraphics[width=.8\linewidth]{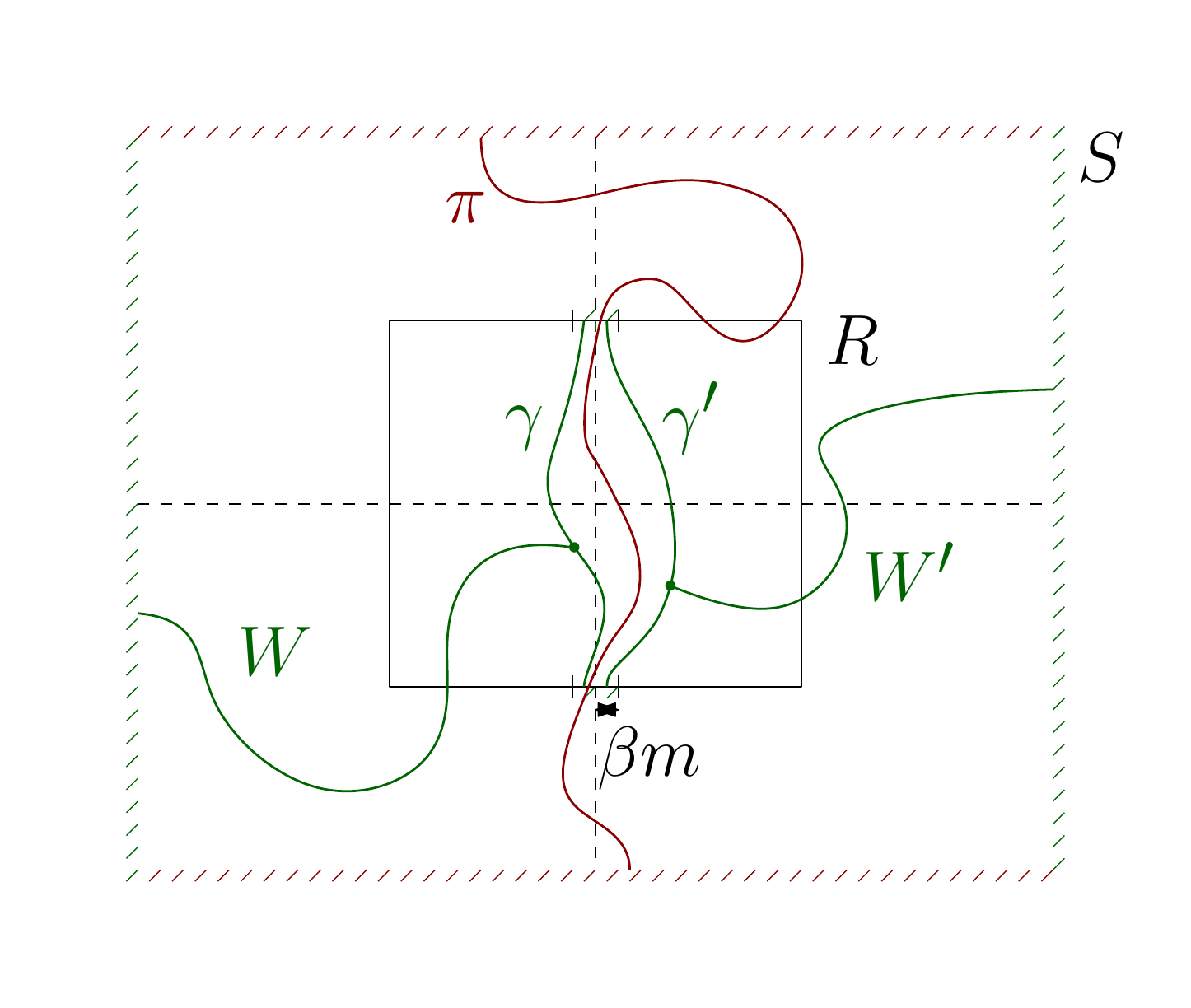}
	\end{minipage}%
	\begin{minipage}{.5\textwidth}
		\centering
		\includegraphics[width=.8\linewidth]{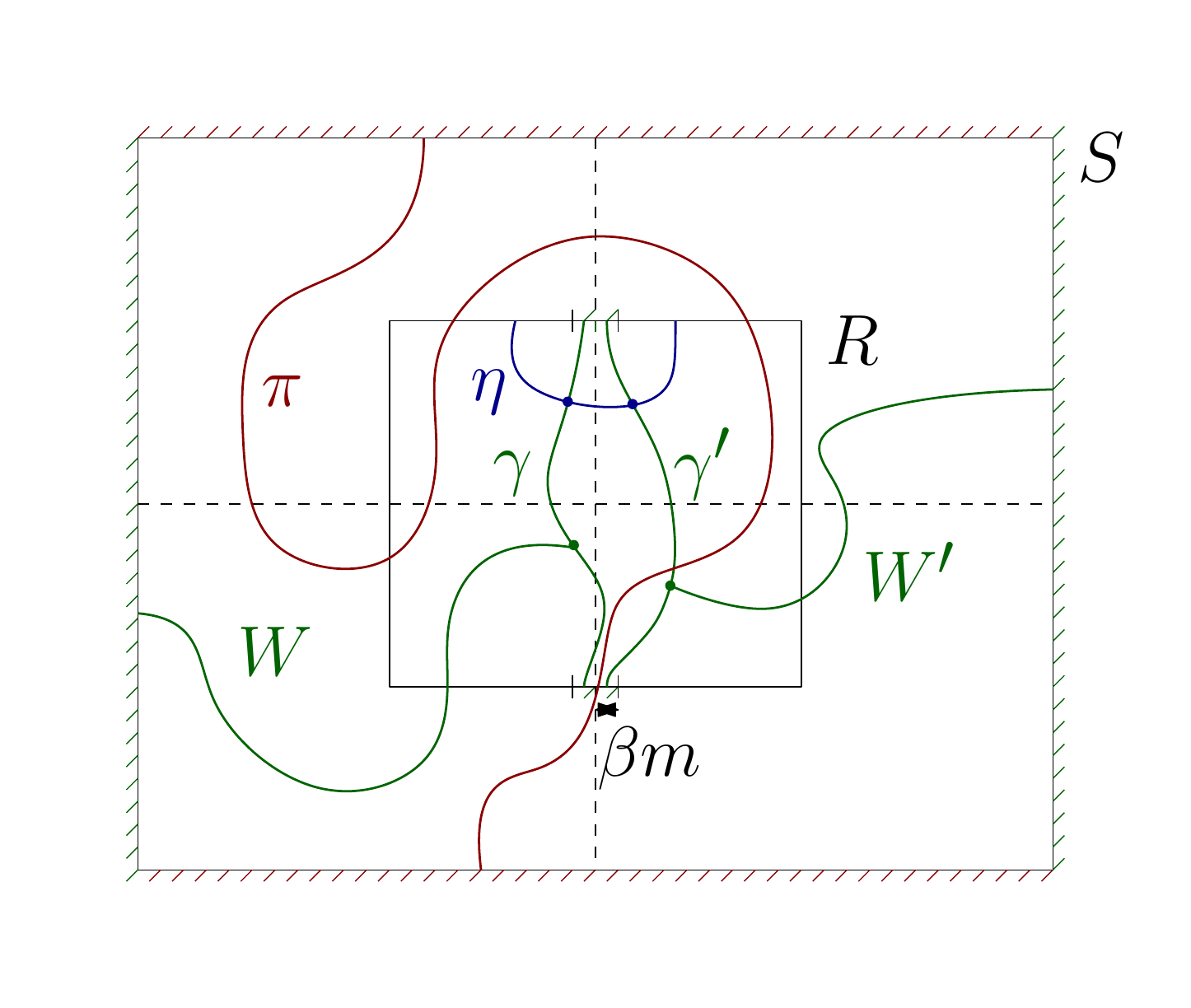}
	\end{minipage}
	\caption{Illustration of the corridor construction in Lemma \ref{lem:1}. Left: the path $\pi$ contains an $m$-path.  Right: the path $\pi$ intersects the structure $W \cup W'$.}
	\label{fig:illustration-corridor-lemma}
\end{figure}

\begin{proof}
   Recall that we see  paths as piecewise linear continuous curves embedded in the plane.    Let $\pi$ be a path from top to bottom and consider the region of the plane defined by
  \begin{equation}
    \label{eq:16}
    B=R\setminus\{\pi\}.
  \end{equation}
  If we assume that $\pi$ contains \emph{no}  $m$-path as a subpath, one can find a continuous curve $\eta$  in $B=R\setminus\{\pi\}$ that connects the left part $\{-m-\beta m\}\times[-m,m]\cup  \{-m-\beta m,0\}\times\{-m,m\}$ to the right part   $\{m+\beta m\}\times[-m,m]\cup  [0,m+\beta m]\times\{-m,m\}$ of $R$ (we emphasize that $\eta$ is  a general continuous planar curve and  we do not require that it corresponds to a lattice path).  Since $\eta$ must intersect both $\gamma$ and $\gamma'$, the set $W\cup\eta\cup W'$ is connected and therefore, it must  contain a continuous curve crossing the rectangle $R$ from left to right. This implies that the path $\pi$ intersect       $W\cup\eta\cup W'$. Since $\pi\cap \eta=\emptyset$ (by definition), this concludes that $\pi$ must intersect $W\cup W'$.     
\end{proof}

\subsection{Construction of quasi-crossings from arm events}
\label{section_proof-of-lemma1}
	\begin{figure}
	\centering
	\begin{minipage}{.5\textwidth}
		\centering
		\includegraphics[width=.75\linewidth]{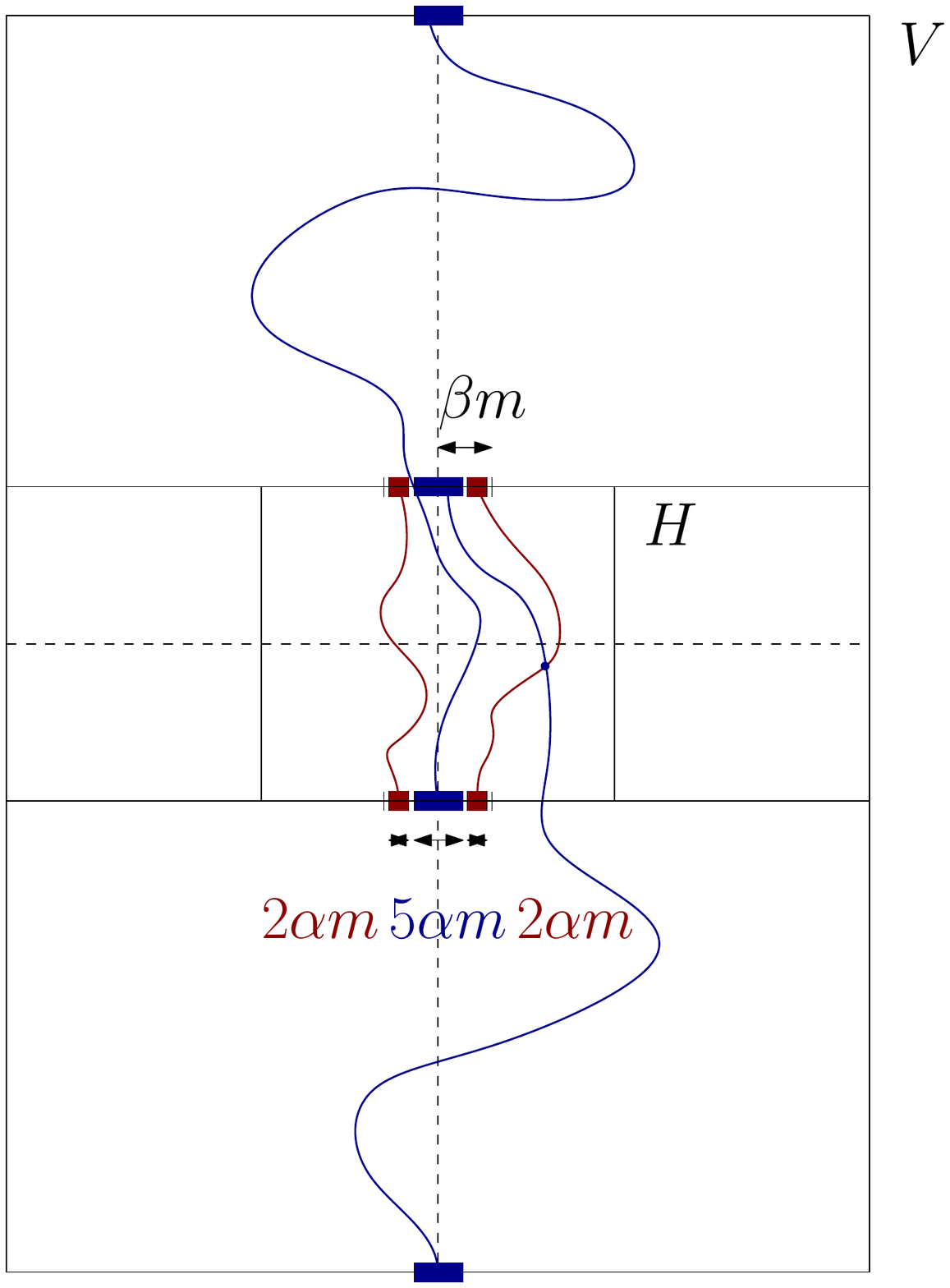}
	\end{minipage}%
	\begin{minipage}{.5\textwidth}
		\centering
		\includegraphics[width=\linewidth]{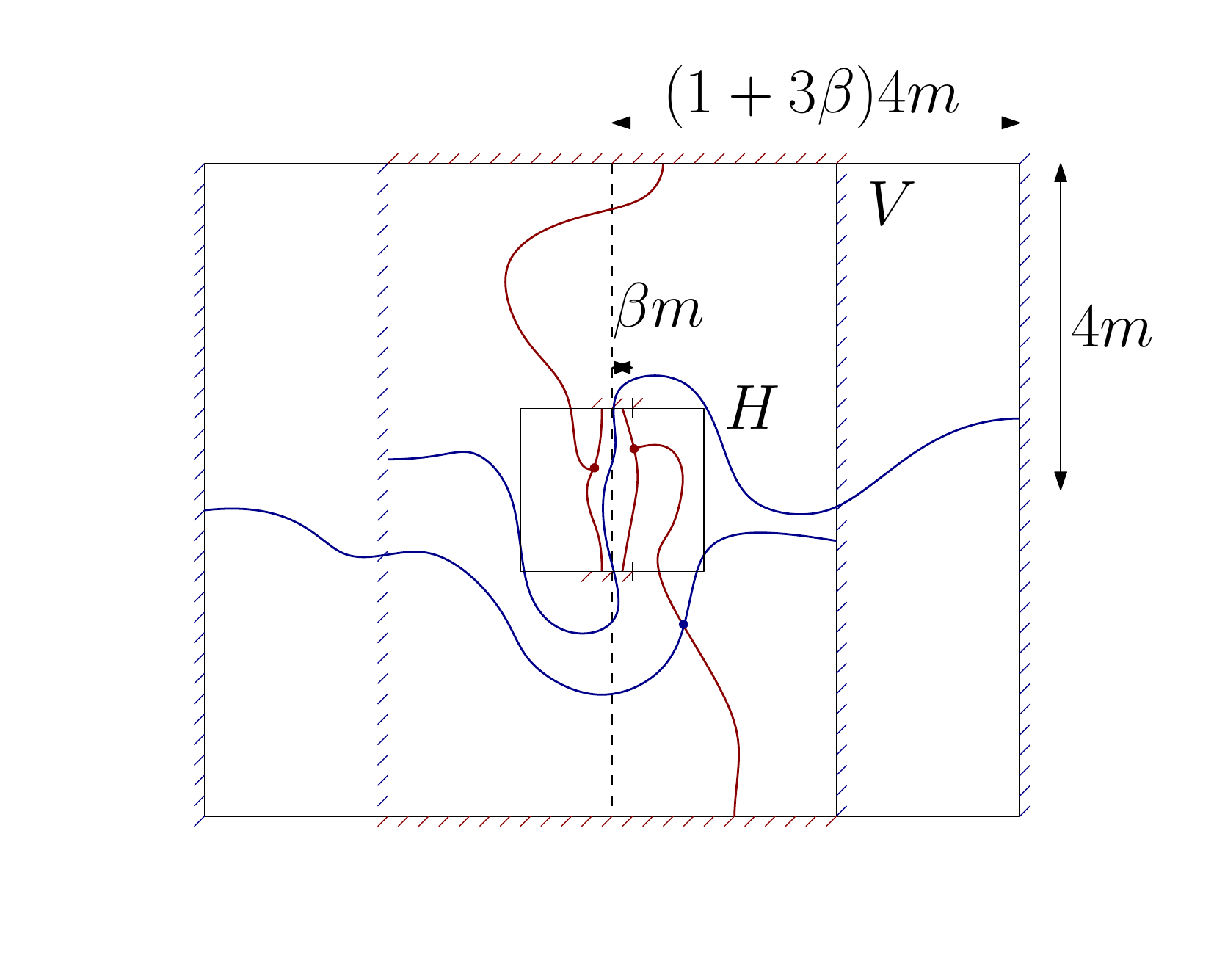}
	\end{minipage}
	\caption{Illustrations of the quasi-crossing construction in Lemma \ref{lemma:second-relation}. Left: intersection of four arm events. Right: intersection of two open $m$-paths that are connected to top and bottom in $V$ (red) with two short crossings from left to right (blue).}
	\label{fig:illustration-lemma1}
      \end{figure}
      In this section,  we construct quasi-crossings using arm events, leading to the following lemma.

      \begin{lemma}
	Let $m=4^j$ for some $j \ge 0$. We have
	\begin{equation*}
	q(4m,m) \ge \left(\min_{\ell \in [m,4m]} a(\ell) \right)^6. 
	\end{equation*}
	\label{lemma:second-relation}
      \end{lemma}
      \begin{remark}
        Notice that the quasi-crossing considered in the lemma above involves two nearby scales. One could try to apply the same strategy to estimate the probability of  quasi-crossings connecting down to a far away scales, but the lower bound would ``degenerate'', and so it would not be useful to obtain a RSW result.
      \end{remark}

\begin{proof}
  For notational convenience, let us consider $m = 4^j$ with $j \ge 2$. The argument for $j \in \{0,1\}$ is identical up to rounding for the choice of translations.
  We begin with constructing two open $m$-paths in the horizontal rectangle $H:=R(m+\beta m,m)$ that are connected to the top side respectively to the bottom side of the vertical rectangle $V:=R(3m-3\beta m,4m)$. To this end (see Figure \ref{fig:illustration-lemma1}), consider the event defined as the intersection of two arm events at scale $m$,
  \begin{equation}\label{eq:27}
    (4\alpha m,0)+\A(m) \quad \text{and} \quad (-4\alpha m,0)+\A(m),
  \end{equation}
  with two arm events at scale $5m/2$,
  \begin{equation}\label{eq:28}
    (0,3m/2)+\A(5m/2) \quad \text{and} \quad (0,-3m/2 )+\A(5m/2).
  \end{equation}
  It is important to notice that the targets of the arms at scale $5m/2$ lie between the targets of the arms at scale $m$. Furthermore, the condition $5\alpha m \le \beta m$ ensures that the two translated arms in \eqref{eq:27} are still $m$-paths. 
  Note that this event implies the existence of two open $m$-paths that are connected to the top and bottom of $V$ respectively. Denoting this event by $\widetilde{\mathcal Q}$, we have
  \begin{equation}
    \label{eq:29}
    \mathbb P[\widetilde{\mathcal Q}]\ge a(m)^2a(5m/2)^2.  
  \end{equation}

  Second, we consider horizontal crossings in the rectangles $[-5m,3m-3\beta m] \times [-4m,4m]$ and $[-3m+3\beta m,5m] \times [-4m,4m]$, whose probabilities are at least $a(31m/8)$ by symmetry. As illustrated in Figure \ref{fig:illustration-lemma1}, intersecting these two horizontal crossings with the event $\widetilde{\mathcal Q}$ guarantees the occurrence of the event $\mathcal Q(4m,m)$. To prove this formally, we can use a variant of the topological result as explained in Remark~\ref{rem:1}. In summary,
  \begin{equation*}
    q(4m,m) \ge  \mathbb P[\widetilde{\mathcal Q}] \cdot a(31m/8)^2 \ge \left(\min_{\ell \in [m,4m]} a(\ell)\right)^6 .
  \end{equation*}
\end{proof}

\subsection{Cascading property of quasi-crossings}
\label{section_proof-of-lemma2}

The following lemma, which relies on the topological result of Lemma~\ref{lem:1}, is the key property of the quasi-crossings. 

\begin{lemma}
	For all $m \ge \ell \ge k \ge 1$, we have
	\begin{equation}\label{eq:14}
	\max \left\{q(m,k),\dfrac{1-\big(1-b(\ell)\big)^{2}}{q(\ell,k)} \right\} \ge 1-\big(1-q(m,\ell)\big)^{1/2}.
	\end{equation}
	\label{lemma:third-relation}
\end{lemma}

\begin{figure}
	\centering
	\begin{minipage}{.5\textwidth}
		\centering
		\includegraphics[width=\linewidth]{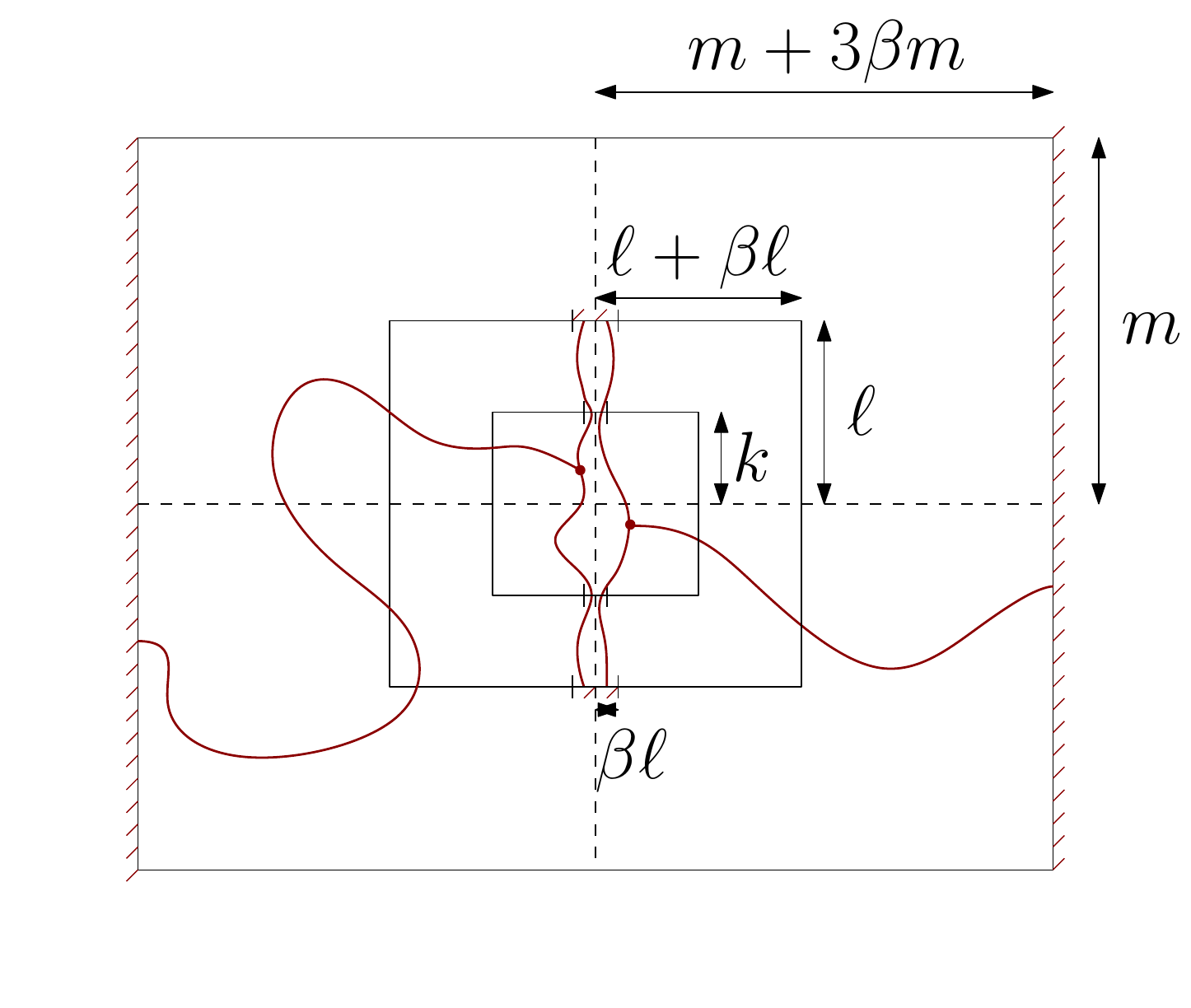}
	\end{minipage}%
	\begin{minipage}{.5\textwidth}
		\centering
		\includegraphics[width=0.9\linewidth]{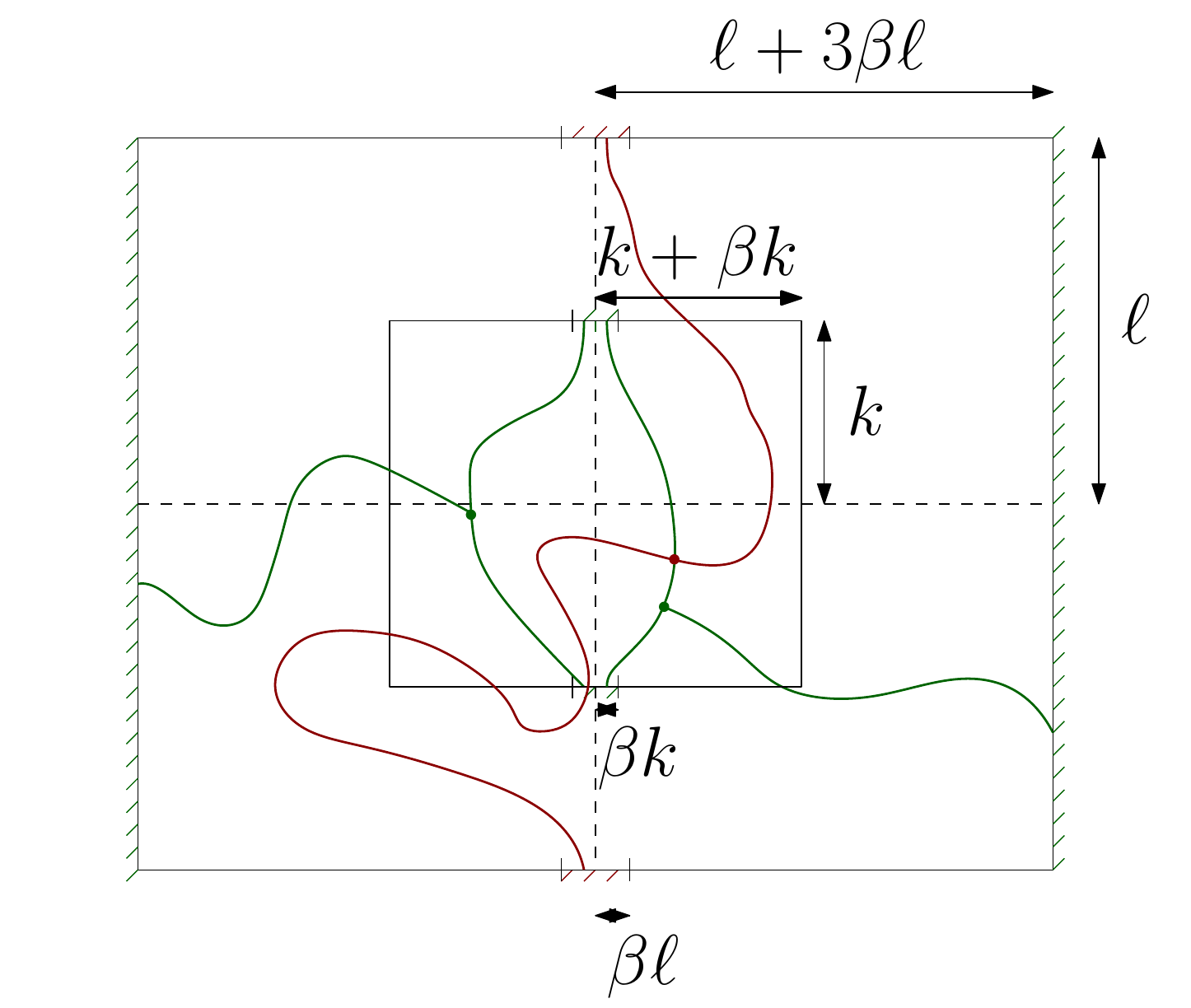}
	\end{minipage}
	\caption{Illustrations of the corridor construction in Lemma \ref{lemma:third-relation}. Left: the event $\mathcal{Q}(m,\ell)$ illustrated with two open $\ell$-paths that both contain a $k$-path. Right: the event $\E \cap \mathcal{Q}(\ell,k)$ with $\gamma$ illustrating an open $\ell$-path that contains \emph{no} $k$-path. }
	\label{fig:illustration-lemma2}
      \end{figure}
      The statement above  can be interpreted as follows: Assume that crossings at scale $m$ reach scale $\ell$, and crossings at scale  $\ell$ reach scale $k$ (for concreteness, say $q(m,\ell)\ge c_0$ and $q(\ell,k)\ge c_1$ for some  constants $c_0,c_1>0$). If the RSW statement does not hold at the intermediate scale $\ell$ (say $b(\ell)<c_0c_1/4$), then   crossings at scale $m$ reach scale $k$ (i.e. $q(m,k)>c_0/2$). It will be important later to notice that the bound we obtain on $q(m,k)$ in this way does not depend on $q(\ell,k)$. 
      
\begin{proof} Let $m \ge \ell \ge k \ge 1$.
  Assume that the quasi-crossing event $\mathcal{Q}(m,\ell)$ occurs. By definition, there exist two open $\ell$-paths that are connected to the left side respectively to the right side of the rectangle $R(m + 3\beta m ,m)$. If each of these two $\ell$-paths contains a $k$-path as a subpath, then the event $\mathcal{Q}(m,k)$ occurs (see Figure \ref{fig:illustration-lemma2}). Otherwise, there exists an open $\ell$-path that contains \emph{no} $k$-path, and we denote this event by $\E$. In summary,
  \begin{equation*}
    \mathcal{Q}(m,\ell) \subset \mathcal{Q}(m,k) \cup \E.
  \end{equation*}
  Note that the event $\mathcal E$ is increasing. To see this, consider the set $\Gamma$ of all $\ell$-paths that contains \emph{no} $k$-path, and notice that $\mathcal E$ is simply the  event that there exists a path in $\Gamma$ which is  open.  The square-root-trick implies
  \begin{equation}
    \max\left\{q(m,k),\P\left[\E\right]\right\} \ge 1-\big(1-q(m,\ell)\big)^{1/2}.\label{eq:26}
  \end{equation}
  We shall have established the lemma if we show that the probability of $\mathcal E$ is smaller than the second term in the max in~\eqref{eq:14}. First, by positive association we have 
  \begin{equation}\label{eq:25}
    \mathbb P[\mathcal E]\le \frac{\mathbb P[\mathcal E\cap \mathcal Q(\ell,k)]}{q(\ell,k)}. 
  \end{equation}
  Now, assume that  $\mathcal E\cap \mathcal Q(\ell,k)$ occurs. By definition, there exist two open $k$-paths $\gamma$ and $\gamma'$ that are connected to the left and  right sides of $S=R(\ell+3\beta \ell,\ell)$ respectively. Furthermore there exists an open $\ell$-path $\pi$ that contains no $k$-path. By the topological result of Lemma~\ref{lem:1}, the path $\pi$ must be  connected to the
 left or to the right side of $S$.  In both cases, a translate of the bridge event $\mathcal B(\ell)$ must occur. More precisely, we obtain
  \begin{equation*}
    \E \cap \mathcal{Q}(\ell,k) \subset \left((-2\beta\ell,0)+\B(\ell)\right) \cup \left((2\beta\ell,0)+\B(\ell)\right).
  \end{equation*}
  We conclude by using positive association and symmetry that
  \begin{equation}
    \label{eq:24}
   \mathbb P[\mathcal E \cap \mathcal{Q}(\ell,k)] \le 1-\mathbb P[ \left((-2\beta\ell,0)+\B(\ell)\right)^c \cap \left((2\beta\ell,0)+\B(\ell)\right)^c]\le 1-(1-b(\ell))^2.
 \end{equation}
 First plugging  this bound into \eqref{eq:25} and then the resulting inequality into \eqref{eq:26} concludes the proof.
\end{proof}

\subsection{Closing property of quasi-crossings}
\begin{figure}
	\centering
	\includegraphics[width=.45\linewidth]{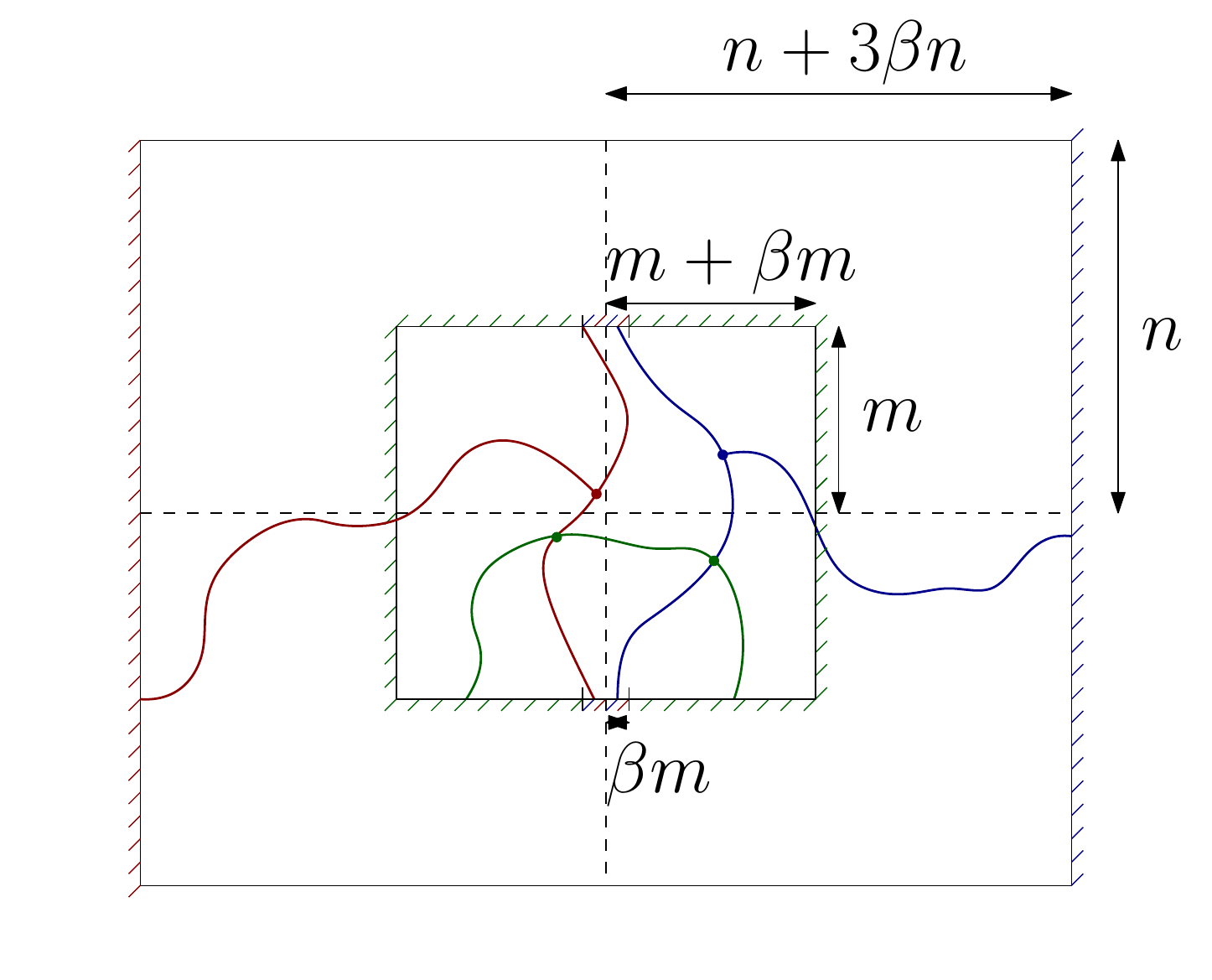}
	\caption{Illustration of Lemma \ref{lemma:fourth-relation}}
	\label{fig:illustration-lemma3}
\end{figure}
Finally, we construct a bridge at scale $n$ by intersecting a quasi-crossing at scale $n$ that connects down to scale $m$, with a bridge at scale $m$. This implies the following relation, whose proof follows readily from positive association (see Figure \ref{fig:illustration-lemma3}).
\begin{lemma}
	For all $m \ge \ell \ge 1$, we have
	\begin{equation*}
		b(m) \ge q(m,\ell) \cdot b(\ell).
	\end{equation*}
	\label{lemma:fourth-relation}
\end{lemma}

\section{Proof of the main theorem}
\label{sec:proof-main-theorem}

In this section, we prove Theorem \ref{thm_full-RSW} based on Lemmas \ref{lemma:first-relation}--\ref{lemma:fourth-relation} and their corresponding dual results.

\subsection{Increasing homeomorphisms}
\label{sec:incr-home}

For integers $i \ge 0$, we define increasing homeomorphisms $f_{i}:[0,1]\to [0,1]$ by
\begin{equation}
	f_{i} (x) = \left(1-(1-x)^{1/2000^i}\right)^{2000^i}.
\end{equation}
Let us collect some elementary properties of these homeomorphisms that will be useful in the proof of Theorem \ref{thm_full-RSW}.

As a direct consequence of the definition, we have the comparison $f_{i} \ge f_{j}$ for $i \le j$. Similarly, for every $x\in [0,1]$ and  $i\ge 0$,
\begin{equation} \label{eq:20}
	f_i(x) \ge \sqrt{f_{i+1}(x)}.
\end{equation}
To exploit duality in the proof of Theorem \ref{thm_full-RSW}, we will use that for every $x,y\in[0,1]$ and $i\ge 0$,
\begin{equation} \label{eq:22}
	x \ge f_i(1-y) \iff y \ge f_i(1-x),
\end{equation}
which follows from the fact that $(1-f_i)\circ (1-f_i) = \textrm{id}$.
Furthermore, for $i\ge 0$, it will be convenient to use the comparison\footnote{Actually, it is possible but maybe less straightforward to prove the more general comparison $f_{i} \circ f_j \ge f_{i+j}$ for all $i,j\ge 0$.} 
\begin{equation}\label{eq:21}
	f_i \circ f_i \ge f_{2i},
\end{equation}
which follows from the inequality $\left(1-(1-x)^{1/a}\right)^a \le 1-\left(1-x^a\right)^{1/a}$ that is valid for  $x \in [0,1]$ and $a\ge 1$. The inequality itself can be proven by elementary means using the change of variable $x'=(1-x)^{1/a}$.

Finally, let us show that for every $x\in[0,1]$,
\begin{equation}
    \label{eq:13}
	\frac{1-\left(1-\!f_1^{-1} \circ f_{18}\big(1-x\big)\right)^2}{f_{4}\left(1- f_1^{-1}\big(x)\big)\right)} 
	\le f_{8}\big(1-x\big).
\end{equation}
 First, note that $ 1-(1-y)^2 \le 1-(1-y^{1/2000})^{2000} = f_1^{-1}(y) $ for $y\in[0,1]$, and so the nominator is bounded from above by 
\begin{equation}
 f_1^{-1}\circ f_1^{-1} \circ f_{18}\big(1-x\big) \le f_2^{-1} \circ f_{18}\big(1-x\big) \le f_{9}\big(1-x\big)\le f_{8}\big(1-x\big)^2,
\end{equation}
where we have used \eqref{eq:21} in the first inequality and \eqref{eq:21} together with $f_2 \ge f_9$ in the second inequality. Second, we note that $1-f_1^{-1}(x)=f_1(1-x) \ge f_4(1-x)$ and so the denominator is bounded from below by $f_4 \circ f_4(1-x) \ge f_8(1-x)$,
where we used again \eqref{eq:21}. Combining the two bounds implies \eqref{eq:13}.

 \subsection{Proof of Theorem~\ref{thm_full-RSW}}
\label{sec:proof-theorem}

	

Let  $\B^\star(n)$ denote the dual bridge event, defined exactly as the bridge event, but replacing the primal connection by a dual connection, and write $b^\star(n)=\mathbb P[\mathcal B^\star(n)]$. Since the dual model satisfies the same hypotheses as the primal one (symmetries and positive association), the relations in Lemmas \ref{lemma:first-relation} and \ref{lemma:second-relation}--\ref{lemma:fourth-relation} hold analogously for dual events. 

We begin with reducing the proof to a relation between bridges and dual bridges. Namely, we claim that it suffices to establish\footnote{By duality, the quantity $1-b^\star(n)$ represents  the probability of a  variant of the primal arm event where the target size is $\beta n$ instead of $\alpha  n$. Therefore, the relation \eqref{eq:reduced-statement} can be interpreted as a a lower bound for the bridge probability in terms of an  arm probability.} 
  \begin{equation}
    b(n) \ge f_{20}\big(1-b^\star(n)\big) \tag{$\ast$}
    \label{eq:reduced-statement}
  \end{equation}
along the sequence $n = 4^i$, $i\ge 0$. To see this, we relate long rectangle crossing probabilities and bridge event probabilities using Lemma \ref{lemma:first-relation}. For $n=4^i$, we have 
\begin{equation}
	\P\left[\C(8n,n)\right] \ge f_1\left(b(n)\right) \ge f_1 \circ f_{20}\left(1-b^\star(n)\right) \ge f_1 \circ f_{20}\left(1- f_1^{-1}\left(\P\left[\C^\star(8 n,n)\right]\right)\right).
\end{equation} 
Then the duality relation~\eqref{eq:7} implies
\begin{equation*}
\P\left[\C\left(8 n, n\right)\right] \ge \psi_2 \big(\P\left[\C\left(n,8 n\right)\right]\big), 
\end{equation*}
where $\psi_2:[0,1]\to[0,1]$ is defined by $\psi_2(x)=f_1 \circ f_{20}(1-f_1^{-1}(1-x))$. For general $n$, we obtain $\P[\C(2 n, n)] \ge \psi_2 (\P[\C(n,2 n)])$ by inclusion of events, which completes the deduction of Theorem \ref{thm_full-RSW} from \eqref{eq:reduced-statement} for $\rho=2$. For general $\rho$, we conclude using the standard gluing construction (see part \textit{(iii)} in  Lemma \ref{lemma:first-relation}).

The remainder of this section is devoted to proving (\ref{eq:reduced-statement}). We consider $n$ of the form $n = 4^i$, $i \ge 0$, and we choose
\begin{equation*}
    m_0 = \max\left\{m=4^j \le n : b(m) \ge f_{18}\big(1-b^\star(m)\big)\right\},
\end{equation*}
the last scale below $n$ at which we have the stronger relation.
The scale $m_0$ is well-defined since $b(1) \ge f_1(1-b^\star(1))$.  To see this, just observe that $b(1)\ge p^2$ and $b^\star(1) \ge (1-p)^3$, where $p$ denotes the probability of a single edge to be open.
\begin{remark} \label{rk:exchanging-roles}
	We emphasize  that   $b$ and $b^\star$ play symmetric roles in the definition of $m_0$,  by  \eqref{eq:22}. As a consequence, the roles of primal and dual can be exchanged in the proof below, which will be important at the end of the proof. 
\end{remark}

If $m_0 = n$, we are already done. Otherwise, we consider the  intermediate scales $m_k:=4^km_0$, $k=0,\ldots ,K$, where  $K\ge 1$ is such that  $n=m_K$, and we show inductively that for all $1\le k\le K$,
  \begin{equation}\label{eq:4}
    q(m_k,m_0) \ge f_4\big(1-b^\star(m_k)\big).
  \end{equation}
First, by Lemma \ref{lemma:second-relation}, we have  $q(m_k,m_{k-1})\ge \min_{\ell \in [m_{k-1},m_k]}a(\ell)^6$ for every $1\le k\le K$, and  by Lemma~\ref{lemma:first-relation}, we have for every $\ell\in [m_{k-1},m_k]$,
\begin{equation}
  \label{eq:5}
  a(\ell)\ge f_1\big(\P[\C(\ell,8\ell)]\big)\ge f_1\big(\mathbb P[\mathcal C(m_k,2m_k)]\big)=f_1\big(1-\mathbb P[\mathcal C^\star(2m_k,m_k)]\big).
\end{equation}
 Using that $b^\star(m_k)\ge\mathbb P[\mathcal C^\star(2m_k,m_k)]$, we  obtain for every $1\le k\le K$, 
\begin{equation}
		q(m_k,m_{k-1}) \ge f_{2}\big(1-b^\star(m_k)\big). \label{eq:19}
\end{equation}
In particular, this proves~\eqref{eq:4} for $k=1$.

Now, let us move to the induction step. Let $2\le k\le K $ and assume that \eqref{eq:4} holds for $k-1$. We use the cascading property of quasi-crossings: by applying Lemma \ref{lemma:third-relation} to the triple $(m_{k},m_{k-1},m_0)$, we obtain 
\begin{equation}
	\max \left\{q(m_k,m_0),\dfrac{1-\big(1-b(m_{k-1})\big)^{2}}{f_{4}\big(1-b^\star(m_{k-1})\big)} \right\} \ge f_{4}\big(1-b^\star(m_k)\big),
\end{equation}
where we have used the bound $q(m_{k-1},m_0) \ge f_{4}(1-b^\star(m_{k-1}))$ from the induction hypothesis on $k-1$ in the denominator and the bound $q(m_k,m_{k-1}) \ge f_{2}(1-b^\star(m_k))$ from \eqref{eq:19} together with the comparison \eqref{eq:21} on the right hand side.
To complete the induction, we observe that the second term in the maximum is bounded by
\begin{align}
	 \frac{1\!-\!\left(1\!-\!b(m_{k-1})\right)^{2}}{f_{4}\big(1-b^\star(m_{k-1})\big)}  \le \frac{1\!-\!\big(1\!-\!f_1^{-1}(b(m_{k}))\big)^2}{f_{4}\left(1- f_1^{-1}\big(b^\star(m_{k})\big)\right)} \le \frac{1\!-\!\left(1\!-\!f_1^{-1} \circ f_i\big(1-b^\star(m_{k})\big)\right)^2}{f_{4}\left(1- f_1^{-1}\big(b^\star(m_{k})\big)\right)} ,
\end{align}
which is strictly smaller than $f_{4}(1-b^\star(m_{k}))$ due to \eqref{eq:13}.
Here, we used Lemma \ref{lemma:first-relation} via $f_1(b(m_{k-1}))\le \mathbb P[\mathcal C(8m_{k-1},m_{k-1})]\le b(m_k)$ and $f_1( b^\star(m_{k-1}))\le b^\star(m_{k})$ in the first inequality, and $b(m_{k}) \le f_i(1-b^\star(m_{k}))$ in the second inequality.

By \eqref{eq:4} applied to $k=K$, we have $q(n,m_0) \ge f_{4}(1-b^\star(n)) \ge \sqrt{f_{18}(1-(b^\star(n))}\ge \sqrt{b(n)}$, using \eqref{eq:20}.  This bound and  Lemma~\ref{lemma:fourth-relation}  imply  $b(n)\ge q(n,m_0)b(m_0)\ge  \sqrt{b(n)}b(m_0)$, and we obtain the following relation between the bridge probabilities at scale $n$  and $m_0$:  
\begin{equation}
  b(n)  \ge  b(m_0)^2.
\end{equation}
Now, as mentioned in Remark \ref{rk:exchanging-roles}, the roles of $b$ and $b^\star$ can be exchanged since the condition $b(m) \ge f_{18}(1-b^\star(m))$ is equivalent to $b^\star(m) \ge f_{18}(1-b(m))$. Hence, we can obtain  the relation  $  b^\star(n)  \ge  b^\star(m_0)^2$ following the same lines as above by using the dual versions of Lemmas~\ref{lemma:first-relation}-\ref{lemma:fourth-relation}.
Combining the relations between bridge probabilities at scale $n$ and $m_0$ for primal and dual with the stronger relation $b(m_0) \ge f_{18}(b^\star(m_0))$ at scale $m_0$ implies 
\begin{equation}
	b(n) \ge b(m_0)^2 \ge f_{18}\left(1-b^\star(m_0)\right)^2 \ge f_{20}\left(1-b^\star(n)\right),
\end{equation}
which completes the proof of \eqref{eq:reduced-statement} and thereby of Theorem \ref{thm_full-RSW}.

\section{Extensions for finite-volume measures}
\label{section_extensions}

Until now, we have considered positively associated measures that are invariant under the symmetries of $\Z^2$. For many possible applications, the latter assumption of full invariance under symmetries is too restrictive and it is particularly unsuitable for  applications to finite-volume measures, where RSW results play an important role. In this section, we therefore drop the assumption of invariance under symmetries and present extensions based on weaker assumptions.

While we have tried to keep the use of invariance under symmetries to a minimum in the proof of Theorem \ref{thm_full-RSW}, we have used it crucially at two places: First, it was important in Lemma \ref{lemma:first-relation} to relate the probabilities of long rectangle crossings and bridge events, and analogously of short rectangle crossings and arm events (via its dual statement). Second, quasi-crossings have been constructed in Section \ref{section_proof-of-lemma1} by intersecting translated and $\pi/2$-rotated versions of arm events at different scales.
In Subsection \ref{subsection_stochastic-domination}, we introduce symmetric stochastic domination and show that it can be used to replace the use of invariance under symmetries in both situations described above. In Subsection \ref{subsection_uniform-bound}, we take a different route by assuming a uniform lower bound on probabilities of arm events. This assumption is clearly sufficient to construct quasi-crossings as in Section \ref{section_proof-of-lemma1}, but it does not enable us to directly relate the probabilities of long crossings and bridge events as in Lemma \ref{lemma:first-relation}. Interestingly, it is nonetheless sufficient to prove a uniform lower bound on probabilities of long crossings.

For $n\ge 1$, we denote by $(\Lambda_n,E_n)$ the subgraph of $(\Z^2,\mathbb{E}^2)$ induced by $[-n,n]^2$. Throughout this section, we consider bond percolation on $(\Lambda_n,E_n)$, i.e.~probability measures on the space of percolation configurations $\{0,1\}^{E_n}$. In order to reuse previously introduced notation, we make the identification between events in $\{0,1\}^{E_n}$ and events in $\{0,1\}^{\mathbb{E}^2}$ that only depend on edges in $E_n$.

\subsection{Symmetric stochastic domination}
\label{subsection_stochastic-domination}

A symmetry $\sigma \in \Sigma$ is called \emph{$n$-admissible for an event $\E$} if $\sigma \boldsymbol{\cdot} \E$ only depends on edges in $E_n$.
Given two probability measures $\P$, $\mathbb{Q}$ on the space of percolation configurations $\{0,1\}^{E_{n}}$, we write $\P \gg_n \mathbb{Q}$ if 
\begin{equation*}
	\P[\sigma\boldsymbol{\cdot}\E] \ge  \mathbb{Q} [\tau\boldsymbol{\cdot}\E]
\end{equation*}
for all increasing events $\E$ and all $n$-admissible $\sigma, \tau \in \Sigma$. Note that this implies $\P[\sigma \boldsymbol{\cdot} \F] \le \mathbb{Q}[\tau \boldsymbol{\cdot} \F]$ for decreasing events $\F$. Within the framework of symmetric stochastic domination, we obtain the following generalization of Theorem \ref{thm_full-RSW}.
\begin{theorem}
	Let  $\P^{+}$, $\P$, $\P^{-}$ be probability measures on  $\{0,1\}^{E_{2n}}$ satisfying positive association and $\P^{+} \gg_{2n} \P \gg_{2n} \P^{-}$. There exists a homeomorphism $\psi : [0,1] \to [0,1]$ such that
	\begin{equation*}
		\P^{+}[\C(2 n,n)] \ge \psi \big(   \P^{-}[\C(n,2 n)]\big).
	\end{equation*} \label{thm_stochastic-domination}
\end{theorem}
As in Theorem \ref{thm_full-RSW}, the homeomorphism $\psi$ is universal: it neither depends on the scale $n$ nor on the probability measures.

Before proving the statement, let us illustrate one application in the  framework of FK-percolation.
For every  finite subset $S\subset \mathbb Z^2$,  consider the FK-percolation measure $\phi_S$  (with fixed edge density $p\in[0,1]$ and cluster weight $q\ge1$) on the graph induced by $S$. This corresponds to the FK-measure in $S$ with free boundary conditions  (see \cite{MR2243761}, or \cite{MR4043224} for background on FK-percolation). Given a symmetry $\sigma\in \Sigma$, it follows directly from the definition of FK-percolation that the measures are \emph{consistent}, in the sense that the measure in $\sigma\cdot S$ coincides with the image of $\phi_S$ under $\sigma$. Furthermore, those measures are \emph{ordered}: \begin{equation}
  \label{eq:31}
  \phi_T(\mathcal A)\ge \phi_S(\mathcal A)
\end{equation}
for every $S\subset T$, and every increasing event $\mathcal A$ in $S$. This corresponds to the idea that the free boundary conditions have a negative effect, and that pushing the free boundary conditions away increases the probability of increasing events (see \cite{duminil2019renormalization} for more details on this phenomena). This implies that for every increasing event $\mathcal A$ depending on the edges in $\Lambda_n$ and every $n$-admissible symmetry $\sigma$ we have
\begin{equation}
  \label{eq:32}
\phi_{\Lambda_{2n}}(\sigma\cdot \mathcal A)\ge \phi_{\sigma\cdot \Lambda_n}(\sigma \cdot \mathcal A) =\phi_{\Lambda_n}(\mathcal A).
\end{equation}
By applying Theorem~\ref{thm_stochastic-domination} to the measures on  $\{0,1\}^{E_{2n}}$ induced by $\phi_{\Lambda_{2n}}$,   $\phi_{\Lambda_{4n}}$ and  $\phi_{\Lambda_{6n}}$, we obtain
\begin{equation}
  \label{eq:33}
  \phi_{\Lambda_{6n}}[\C(2 n,n)] \ge \psi \big( 	\phi_{\Lambda_{2n}}[ \C(n,2n)] \big).
\end{equation}
In other words,  up to pushing slightly the boundary conditions, short and long crossings  for FK-percolation behave the same way. This was already known for FK-percolation using different methods (see Proposition 8 in \cite{duminil2019renormalization}), and Theorem~\ref{thm_stochastic-domination} above generalizes this result to any  consistent, ordered family of positively associated measures.  

\begin{proof}[Proof of Theorem~\ref{thm_stochastic-domination}]
	Without loss of generality, we consider $n$ of the form $n=4^i$. Throughout the proof, all symmetries will be assumed to be $n$-admissible for the respective events. We claim that it suffices to prove for any $\sigma,\tau \in \Sigma$,
	\begin{equation}
		\tag{$\ast\ast$} \label{eq:reduced-statement-stochastic-domination}
		\P^{+}\left[\sigma \boldsymbol{\cdot} \C(n + 3\beta n,n)\right] \ge f_{20}\left(1-\P^{-}\left[\tau \boldsymbol{\cdot} \C^\star(n+ 3\beta n,n)\right]\right).
	\end{equation}
	Indeed, to deduce Theorem \ref{thm_stochastic-domination} from this statement, we note that the standard gluing of long rectangle crossings (as in part \textit{(iii)} of Lemma \ref{lemma:first-relation}) implies 
	\begin{equation*}
		\P^{+}[\C(2n,n)]\ge (\min_\sigma \P^{+}[\sigma \boldsymbol{\cdot} \C(n+3\beta n,n)])^{7}
	\end{equation*} 
	as well as the analogous inequality for long dual crossings. This yields the theorem with  the universal homeomorphism $\psi = f_{21}$.
	
	To prove \eqref{eq:reduced-statement-stochastic-domination}, we follow Section \ref{sec:proof-main-theorem} closely and begin with stating Lemmas \ref{lemma:second-relation}--\ref{lemma:fourth-relation} in the framework of symmetric stochastic domination. For $m \in [1,n]$, we denote
	\begin{equation*}
		\ubar{a}(m):=\min_{\sigma} \P\left[\sigma \boldsymbol{\cdot} \A(m)\right], \quad \bar{b}(m):=\max_{\sigma} \P\left[\sigma \boldsymbol{\cdot} \B(m)\right],
	\end{equation*}
	and as before for $\ell \le m$, $q(m,\ell)=\P[\Q(m,\ell)]$. 
	As in Sections \ref{section_proof-of-lemma1} and \ref{section_proof-of-lemma2}, it holds that for $m=4^j \in [1,n/4]$,
	\begin{equation}
		\label{eq:lemma1-p} \tag{L.3}
		q(4m,m) \ge \left(\min_{\ell \in [m,4m]} \ubar{a}(\ell) \right)^6, 
	\end{equation}
	and for $ m \ge \ell \ge k \in [1,n]$,
	\begin{equation}
		\tag{L.4} \label{eq:lemma2-p}
		\max \left\{q(m,k),\dfrac{1-\big(1-\bar{b}(\ell)\big)^{2}}{q(\ell,k)} \right\} \ge 	1-\big(1-q(m,\ell)\big)^{1/2}.
	\end{equation} 
	Third, positive association implies for $\ell \le m \in [1,n]$ and for any $\sigma \in \Sigma$,
	\begin{equation}
		\tag{L.5} \label{eq:lemma3-p}
		\P\left[\sigma \boldsymbol{\cdot} \C(m + 3\beta m,m)\right] \ge \P[\sigma\boldsymbol{\cdot}\Q(m,\ell)]  \cdot \P[\sigma \boldsymbol{\cdot} \B(\ell))],
	\end{equation}
	and the same holds true for the measures $\P^{+}$ and $\P^{-}$.

	We choose $m_0$ to be the maximal $m = 4^j \le n$ such that for all $\sigma, \tau \in \Sigma$, 
	\begin{equation*}
		\P^{+}\left[\sigma \boldsymbol{\cdot} \C(m + 3\beta m,m)\right] \ge f_{18} \left(1-\P^{-}\left[\tau \boldsymbol{\cdot} \C^\star(m + 3\beta m,m)\right]\right).
	\end{equation*}
	If $m_0 = n$, the desired relation \eqref{eq:reduced-statement-stochastic-domination} holds true. Otherwise, it can be shown inductively that for all $m_k := 4^k m_0 \in [4m_0,n]$ and for any $\sigma \in \Sigma$,
	\begin{equation*}
		q(m_k,m_0) \ge f_4\left(1-\P^{-}\left[\sigma \boldsymbol{\cdot} \C^\star(m_k + 3\beta m_k,m_k) \right]\right).
	\end{equation*}
	As in the proof of Theorem \ref{thm_full-RSW},  the argument relies on \eqref{eq:lemma1-p} and \eqref{eq:lemma2-p} and the details are left to the reader. In particular for any $\sigma, \tau \in \Sigma$,
	\begin{equation*}
		\P^{+}\left[\sigma \boldsymbol{\cdot} \Q(n,m_0) \right] \ge f_4\left(1-\P^{-}\left[\tau \boldsymbol{\cdot} \C^\star(n+3\beta n,n) \right]\right) \ge \sqrt{\P^{+}\left[\sigma \boldsymbol{\cdot} \C(n + 3\beta n, n) \right]},
	\end{equation*}
	and so for any $\sigma \in \Sigma$, the probabilities of long crossings at scales $n$ and $m_0$ are related by 
	\begin{equation*}
		\P^{+}\left[\sigma \boldsymbol{\cdot} \C(n+3\beta n,n)\right] \ge \P^{+}\left[\sigma \boldsymbol{\cdot} \C(m_0 + 3\beta m_0,m_0)\right]^2 ,
	\end{equation*}
	where we have used \eqref{eq:lemma3-p}. The analogous relation holds for dual long crossings under the measure $\P^{-}$, and so the desired relation \eqref{eq:reduced-statement-stochastic-domination} follows thanks to the choice of $m_0$. 
\end{proof}

\subsection{Uniform bound on arm events}
\label{subsection_uniform-bound}
In this section, we show how to obtain a uniform lower bound on long crossing probabilities from a uniform lower bound on arm events. 
For $m \ge 1$, we introduce  a family of symmetries allowing only for translations to distance at most $m$. We write
\begin{align*}
	\Sigma_m:= &\left\{ \sigma \in \Sigma  : \sigma \boldsymbol{\cdot} \Lambda_m \subset \Lambda_{2m} \right\}.
\end{align*} 

\begin{theorem}
	\label{thm_uniform-bound}
	Let $n=4^i$ for $j\ge0$, and let $\P$ be a probability measure on $\{0,1\}^{E_{3n}}$ satisfying positive association. If for some $c>0$ and for all scales $m \le n$,
	\begin{equation*}
		\min_{\sigma \in \Sigma_m} \P\left[\sigma \boldsymbol{\cdot} \A(m)\right] \ge c,
	\end{equation*}
	then there exists a constant $c'=c'(c) > 0$ such that
	\begin{equation*}
		\P\left[ \C(n+3\beta n,n)\right] \ge c'.
	\end{equation*}	
\end{theorem}

Theorem \ref{thm_uniform-bound} highlights that bounding the probability of $\C(n + 3\beta n,n)$ only involves the use of translated and rotated versions of arm events $\A(m)$ that are at distance at most $m$ from the origin. As a direct corollary of Theorem \ref{thm_uniform-bound}, one can obtain uniform bounds on long crossing probabilities for any aspect ratio by considering larger families of symmetries.

It is natural to ask if the uniform bound on arm event probabilities can be replaced by a uniform bound on short crossing probabilities. A simple counterexample shows that the answer is negative (see Figure \ref{fig:counterexample}): consider a deterministic model of open diagonals, where horizontal edges $\{(x,y),(x+1,y)\}$ are open if and only if $x+y$ is even, and vertical edges $\{(x,y),(x,y+1)\}$ are open if and only if $x+y$ is odd.

\begin{proof}[Proof of Theorem \ref{thm_uniform-bound}]
	In the proof, we will introduce constants $c_0>c_1>c_2>c_3>c_4>0$ which by convention do not depend on  the considered scales.	To start with, we state versions of Lemmas \ref{lemma:second-relation}--\ref{lemma:fourth-relation} based on the bound
	\begin{equation*}
		\min_{m\le n} \, \min_{\sigma \in \Sigma_m} \P\left[\sigma \boldsymbol{\cdot} \A(m)\right] \ge c_0. 
	\end{equation*}
	First, quasi-crossings can be constructed as in Section \ref{section_proof-of-lemma1} with uniformly bounded probability. For $m \le n/4$,
	\begin{equation}
		\tag{L.3'} \label{eq:lemma1-pp}
		q(4m,m)  \ge c_1 := {c_0}^6.
	\end{equation}
	Second, we obtain for $m = 4^j \in [1,n/4]$ and $ \ell \le m$,
	\begin{equation}
		\tag{L.4'} \label{eq:lemma2-pp}
		\max\left\{q(4m,\ell),\frac{b(4m)}{q(m,\ell) \cdot {c_1}^2 }\right\} \ge c_2 := c_1 /2.
	\end{equation} 
	To see this, we recall the proof of Lemma \ref{lemma:third-relation} for the triple $(4m,m,\ell)$: By a union bound and \eqref{eq:lemma1-pp}, it holds that $\max\left\{q(4m,\ell),\P\left[\E\right]\right\}$ is at least $c_1 /2$, where $\E$ denotes the existence of an open $m$-path that contains no $\ell$-path as a subpath. As argued in the proof of Lemma \ref{lemma:third-relation}, it holds that
	\begin{equation*}
		\E \cap \Q(m,\ell) \subset \left((-m/6,0)+\B(m)\right) \cup \left((m/6,0)+\B(m)\right).
	\end{equation*}
	It is then easy to check that 
	\begin{equation*}
		\E \cap \Q(m,\ell) \cap  \left((-m/6,0)+\Q(4m,m)\right) \cap \left((m/6,0)+\Q(4m,m)\right) \subset \B(4m),
	\end{equation*}
	which readily implies \eqref{eq:lemma2-pp}. Here, we use that the event $\Q(4m,m)$ is only translated to distance $m/6$ so that its probability can be bounded similarly to \eqref{eq:lemma1-pp}.
	Third, positive association implies that for $m \ge \ell \ge 1$,
	\begin{equation}
		\tag{L.5'} \label{eq:lemma3-pp}
		b(m) \ge  q(m,\ell) \cdot b(\ell).
	\end{equation}
	
	We are now ready to prove that for a sufficiently small constant $c_4 >0$,
	\begin{equation*}
		b(n/4) \ge c_4.
	\end{equation*}
	To this end, we note that $b(1) \ge {c_0}^2$, and so the choice
	\begin{equation*}
		m_0 = \max\left\{m = 4^j \le n/4 : b(m) \ge c_3 := (c_1 c_2)^2\right\}
	\end{equation*}
	is  well-defined. If $m_0 = n/4$, there is nothing to prove. Otherwise, it inductively follows from \eqref{eq:lemma1-pp} and \eqref{eq:lemma2-pp} that for all $m_k:=4^k m_0 \in [4m_0,n/4]$,
	\begin{equation*}
		q(m_k,m_0) \ge c_2.
	\end{equation*}
	Hence, we deduce from \eqref{eq:lemma3-pp} that 
	\begin{equation*}
		b(n/4) \ge q(n/4,m_0)\cdot b(m_0) \ge c_2 c_3 =:c_4
	\end{equation*}
	as desired.
	Finally, it suffices to note that
	\begin{equation*}
		\P[\C(n+3\beta n,n)]\ge q(n,n/4) \cdot b(n/4) \ge c_1 c_4,
	\end{equation*}
	which concludes the proof.
\end{proof}


\small
\bibliographystyle{alpha}
\bibliography{refs}

\end{document}